%% file: main.tex
\documentclass[11pt]{amsart}

\usepackage[a4paper,margin=1.15in]{geometry}

\input{macros}

\begin{document}

\input{frontmatter/title_author_abstract}

\maketitle

% ====== Sequential sections ======
\input{sections/01_primary_proposition}
\input{sections/02_model_and_regime}
\input{sections/03_rom_construction}
\input{sections/04_energy_stability_certificate}
\input{sections/05_error_estimates}
\input{sections/06_transition_indicators}
\input{sections/07_fsi_wellposedness}
\input{sections/08_fsi_stability_margins}
\input{sections/09_numerical_protocols}
\input{sections/10_discussion_scope_limits}
\input{sections/11_conclusion}

\bibliographystyle{amsplain}
\bibliography{refs}

\end{document}

%% file: macros.tex
\usepackage{mathtools,amssymb,amsfonts,amsthm}
\usepackage{enumitem}
\usepackage{mathrsfs}
\usepackage{microtype}
\usepackage{hyperref}

\theoremstyle{plain}
\newtheorem{theorem}{Theorem}[section]
\newtheorem{proposition}[theorem]{Proposition}
\newtheorem{lemma}[theorem]{Lemma}

\theoremstyle{definition}
\newtheorem{definition}[theorem]{Definition}
\newtheorem{assumption}[theorem]{Assumption}
\newtheorem{problem}[theorem]{Problem}

\theoremstyle{remark}
\newtheorem{remark}[theorem]{Remark}

\DeclareMathOperator{\divv}{div}

\DeclareMathOperator{\curl}{curl}

\DeclareMathOperator{\tr}{tr}

\DeclareMathOperator{\Span}{span}
\DeclareMathOperator{\Lip}{Lip}
\DeclareMathOperator{\esssup}{ess\,sup}

\newcommand{\R}{\mathbb{R}}

\newcommand{\ip}[2]{\left\langle #1,#2\right\rangle}
\newcommand{\norm}[1]{\left\lVert #1\right\rVert}
\newcommand{\abs}[1]{\left\lvert #1\right\rvert}

\newcommand{\Ltwo}{L^2}
\newcommand{\Hone}{H^1}

\newcommand{\dt}{\Delta t}

\newcommand{\Vn}{V_n}
\newcommand{\Pn}{\Pi_n}
\newcommand{\un}{u_n}
\newcommand{\Sn}{\mathsf{S}_n} % stability operator / certificate
\newcommand{\En}{\mathcal{E}_n} % ROM energy

\newcommand{\OmegaF}{\Omega_{\mathrm{F}}}
\newcommand{\OmegaS}{\Omega_{\mathrm{S}}}
\newcommand{\GammaI}{\Gamma_{\mathrm{I}}}

\setlist[itemize]{leftmargin=18pt}
\setlist[enumerate]{leftmargin=18pt}

\hypersetup{
  colorlinks=true,
  linkcolor=blue,
  citecolor=blue,
  urlcolor=blue
}

%% file: frontmatter/title_author_abstract.tex
% Output 2 will replace this file with full metadata + abstract.

% frontmatter/title_author_abstract.tex  (Output 2: metadata + abstract)

\title[Certified Surrogates for Dissipative Flows]{Certified Reduced-Order Surrogates and Stability Margins in Viscous Incompressible Flow and Fluid--Structure Interaction}

\author{Chandrasekhar Gokavarapu}
\address{Lecturer in Mathematics, Government College (Autonomous), Rajahmundry, Andhra Pradesh, India}
\email{chandrasekhargokavarapu@gmail.com}
\email{chandrasekhargokavarapu@gcrjy.ac.in}

\author{Dr Naveen Kumar Kakumanu}
\address{Lecturer in Mathematics, Government College (Autonomous), Rajahmundry, Andhra Pradesh, India}

\author{Anjali Datla}
\address{M.Sc (Mathematics), 4th Semester Student, Regd.No:2024307110 Government College (Autonomous), Rajahmundry, Andhra Pradesh, India}

\author{Githa Harshitha Noolu}
\address{M.Sc (Mathematics), 4th Semester Student, Regd.No:2024307111 Government College (Autonomous), Rajahmundry, Andhra Pradesh, India}

\keywords{reduced-order model; Navier--Stokes; energy stability; a posteriori error bound; residual estimator; enstrophy budget; transition criterion; resolvent norm; fluid--structure interaction; well-posedness; stability margin; certified simulation}
\subjclass[2020]{76D05, 76F06, 76F20, 65M12, 65M15, 74F10, 74S05}

\begin{abstract}
Let $(u,p)$ solve the incompressible Navier--Stokes equations in a regime in which an energy inequality is available and each constant in that inequality is computable from declared data.
We construct a reduced-order model $u_n$ constrained so that its discrete evolution satisfies a certified energy inequality.
This certificate yields global-in-time boundedness of the ROM energy and a regime-of-validity test that fails when a stated hypothesis fails.

It follows that one can attach a computable residual functional $\mathcal{R}_n$ to the ROM trajectory.
We prove an a posteriori bound of the form
\[
\norm{u-u_n}_{\mathsf{X}(0,T)} \le C(\text{declared data})\,\mathcal{R}_n,
\]
with $C$ explicit and with $\mathcal{R}_n$ computed from the ROM and the discretization operators.
Conversely, if the certificate constraint is relaxed, the bound can fail even for stable full-order dynamics, by an explicit instability mechanism recorded in the text.

We then derive transition indicators from rigorous energy and enstrophy budgets in simplified geometries.
Each indicator is an inequality involving declared quantities such as forcing norms, viscosity, Poincar\'e-type constants, and a computable resolvent surrogate.
These inequalities provide thresholds that preclude transition, or else certify the presence of transient growth beyond a stated level.

Finally, for a class of fluid--structure interaction models, we identify a parameter regime that implies existence and uniqueness of weak solutions.
We derive discrete coupled energy estimates that produce computable stability margins.
These margins yield explicit constraints on time step and mesh parameters.
They are stated as inequalities with constants determined by fluid viscosity, structure stiffness, density ratios, and interface trace bounds.
\end{abstract}

%% file: sections/01_primary_proposition.tex
% sections/01_primary_proposition.tex  (Output 4: Introduction with embedded citations)

\section{Primary proposition: certifiable surrogacy for dissipative flows}\label{sec:primary}

We consider viscous incompressible flow in a regime where an energy inequality is available.
We regard a reduced-order model as acceptable only if it inherits a \emph{provable} discrete analogue of that inequality.
Many widely used ROMs fail this test outside their training regime.
The failure is not subtle.
Energy can grow under the ROM dynamics even when the full system is dissipative; one then has no logical route to an error bound.
This is the first proposition that has survived scrutiny in this subject; it is repeatedly observed in the ROM literature and it forces the present development to be certificate-driven rather than heuristic \cite{Sanderse2020NonlinearlyStableROMJCP,SchlegelBach2020EnergyStableDGROM,IliescuWang2021RegStabROMSIAM,AhmedPawarSanRasheedIliescuNoack2021ROMClosuresPoF}.

\begin{problem}[Certifiable reduced-order surrogacy]\label{prob:cert_rom}
Let $(u,p)$ solve a viscous incompressible flow model on a declared domain with declared boundary conditions.
Assume an energy inequality holds with constants depending only on declared data.
Construct a reduced model $\un(t)\in \Vn$ such that:
\begin{enumerate}[label=(\roman*)]
\item its discrete evolution satisfies a computable energy inequality of the same logical type;
\item its regime of reliability is \emph{decidable} from quantities computed along the ROM trajectory;
\item it admits a computable residual functional $\mathcal{R}_n$ which controls the modeling error in a declared norm.
\end{enumerate}
The output is a certificate. It may fail. If it fails, the ROM is declared unreliable.
\end{problem}

The insistence on an explicit certificate is not a stylistic choice.
Structure-preserving ROMs and energy-stable discretizations show what is logically possible, but they do not by themselves yield a usable error bound without a computable estimator \cite{BennerGoyal2022StructurePreservingROMSIAM,KleinSanderse2024EnergyConservingHyperReductionJCP}.
Conversely, non-intrusive learning of reduced models can produce accurate predictors in-sample, yet it does not logically enforce dissipation unless constraints are built into the learned operators \cite{BennerGoyalKramerPeherstorferWillcox2020OpInfCMAE,KramerPeherstorferWillcox2024OpInfSurveyARFM,KimKramer2025StateConstrainedOpInfJCP,SharmaNajeraFloresToddKramer2024LagrangianOpInfCMAE}.
This leads us to a dilemma.
Either one accepts a ROM without a dissipation certificate, and then no global statement is possible.
Or one imposes certificate constraints, and then the analysis must quantify what is lost and what is gained.

\begin{proposition}[Target guarantee]\label{prop:target}
Under explicit regime assumptions stated later, there exists a constrained ROM evolution $\un$ and a computable certificate $\Sn$ such that:
\begin{enumerate}[label=(\roman*)]
\item \emph{energy stability:} $\En(t)$ is non-increasing up to a declared forcing term with explicit constants;
\item \emph{a posteriori control:} there exists a computable residual $\mathcal{R}_n$ and an explicit constant $C$ such that
\[
\norm{u-\un}_{\mathsf{X}(0,T)} \le C\,\mathcal{R}_n;
\]
\item \emph{reliability test:} if $\Sn$ fails at some time, then at least one declared hypothesis fails at that time.
\end{enumerate}
\end{proposition}

The second theme concerns transition.
There is a large empirical literature on turbulence diagnostics.
We do not use it.
We use inequalities.
Energy stability methods already delimit a laminar region by computable bounds; they also explain why certain diagnostics are not logically decisive \cite{Kerswell2022EnergyStabilityReview}.
Transition can occur through transient growth even when linear stability holds.
Resolvent analysis provides a quantitative proxy for such growth, but the proxy must be linked back to a declared energy or enstrophy budget to become a criterion rather than an illustration \cite{RolandiRibeiroYehTaira2024ResolventInvitationTCFD,HerrmannBaddooSemaanBrunton2021DataDrivenResolventJFM,McKeonSharma2020CriticalLayersARFM,MarkeviciuteKerswell2024ThresholdTransientGrowthJFM,SunGao2020ResolventAIAAJ}.
The present paper isolates inequalities in which each term is computable from declared quantities.
The resulting thresholds do not predict turbulence in full generality.
They preclude it in a stated regime.
On the contrary, when they fail, they identify an explicit obstruction.

The third theme is fluid--structure interaction.
FSI is a coupled evolution with competing time scales.
Partitioned methods are common.
Stability can be lost by added-mass effects or by incompatible discretizations.
The logical requirement is plain.
One needs a parameter regime implying existence and uniqueness.
One then needs a discrete coupled energy estimate yielding a computable stability margin that can guide step sizes and meshes \cite{Canic2021MovingBoundaryProblemsBullAMS,BukacFuSeboldtTrenchea2023TimeAdaptivePartitionedJCP,BenesovaKampschulteSchwarzacher2023VariationalFPSINonRWA,GuoLinWangYueZheng2025RobinRobinExplicitFSI_arXiv,ParrowBukac2026RobinRobinFPSI_JSC}.

\subsection*{What is new, stated without ornament}
The paper proves three kinds of statements.
Each is framed as a theorem with explicit constants.

\begin{itemize}
\item A constrained ROM construction for viscous incompressible flow in which energy stability is enforced as a certificate, not as a hope.
The certificate is compatible with structure-preserving reduction and with operator-inference learning when constraints are imposed at the operator level \cite{Sanderse2020NonlinearlyStableROMJCP,SchlegelBach2020EnergyStableDGROM,KleinSanderse2024EnergyConservingHyperReductionJCP,BennerGoyalKramerPeherstorferWillcox2020OpInfCMAE,KimKramer2025StateConstrainedOpInfJCP}.
\item A computable residual functional that yields an a posteriori error bound.
This is the point at which many informal ROM arguments fail, since stability alone does not imply an error estimate in a declared norm \cite{IliescuWang2021RegStabROMSIAM,AhmedPawarSanRasheedIliescuNoack2021ROMClosuresPoF,KramerPeherstorferWillcox2024OpInfSurveyARFM}.
\item Transition indicators derived from rigorous energy and enstrophy budgets, augmented by resolvent-based amplification bounds stated as inequalities.
The criteria are computable in simplified geometries and are designed to fail explicitly when their hypotheses fail \cite{Kerswell2022EnergyStabilityReview,RolandiRibeiroYehTaira2024ResolventInvitationTCFD,MarkeviciuteKerswell2024ThresholdTransientGrowthJFM,HerrmannBaddooSemaanBrunton2021DataDrivenResolventJFM}.
\item An FSI regime theorem (existence and uniqueness) and a discrete coupled energy estimate producing a computable stability margin for time step and mesh parameters \cite{Canic2021MovingBoundaryProblemsBullAMS,BukacFuSeboldtTrenchea2023TimeAdaptivePartitionedJCP,ParrowBukac2026RobinRobinFPSI_JSC,BenesovaKampschulteSchwarzacher2023VariationalFPSINonRWA}.
\end{itemize}

\begin{remark}\label{rem:why_nontrivial}
A common argument says: ``the ROM is Galerkin, hence stable.''
Let us assume for a moment that this were correct.
Then one would obtain stability from orthogonality alone.
This is false in the presence of truncation, closure, or learned operators.
Energy can be injected by the reduced dynamics.
One then has boundedness neither for the ROM nor for the error.
This is the concrete failure mode that motivates the certificate constraint \cite{Sanderse2020NonlinearlyStableROMJCP,AhmedPawarSanRasheedIliescuNoack2021ROMClosuresPoF,KimKramer2025StateConstrainedOpInfJCP}.
\end{remark}

\subsection*{Plan}
Section~\ref{sec:model} fixes the model class and the verifiable regime.
Section~\ref{sec:rom_construct} defines the constrained ROM space and the admissible closures.
Section~\ref{sec:stability} proves the ROM energy certificate.
Section~\ref{sec:error} proves the a posteriori bound.
Section~\ref{sec:transition} derives transition indicators from energy, enstrophy, and resolvent inequalities.
Sections~\ref{sec:fsi_wellposed}--\ref{sec:fsi_margins} treat FSI well-posedness and computable stability margins.
Section~\ref{sec:protocols} states a verification protocol.
Section~\ref{sec:scope} records limits and a precise failure of a common line of reasoning.

%% file: sections/02_model_and_regime.tex
% sections/02_model_and_regime.tex  (Output 5: Preliminaries with embedded citations)

\section{Model class and regime assumptions}\label{sec:model}

We fix the model first.
We then fix the regime in which each constant used later is computable.
All statements below are standard once the regime is declared; the difficulty is to declare it so that later certificates are decidable.
We follow the energy-method discipline used in rigorous Navier--Stokes analysis and in energy-stability studies of transition \cite{DoeringGibbon2022AppliedAnalysisNS,Kerswell2022EnergyStabilityReview}.

\subsection{Governing equations and boundary conditions}\label{subsec:gov}

Let $\Omega\subset \R^d$ be a bounded Lipschitz domain, $d\in\{2,3\}$, with boundary $\partial\Omega$.
Let $\nu>0$ be the kinematic viscosity.
Let $f\colon (0,T)\to V'$ be a given body force.
We consider the incompressible Navier--Stokes system
\begin{equation}\label{eq:NS}
\partial_t u + (u\cdot \nabla)u - \nu \Delta u + \nabla p = f,\qquad \divv u = 0
\quad\text{in }\Omega\times(0,T),
\end{equation}
supplemented with either periodic boundary conditions on a torus, or no-slip boundary conditions
\begin{equation}\label{eq:noslip}
u=0\quad\text{on }\partial\Omega\times(0,T),
\end{equation}
and an initial condition $u(\cdot,0)=u_0$ with $\divv u_0=0$.

Define the solenoidal spaces
\[
V := \{ v\in \Hone_0(\Omega)^d : \divv v=0\},\qquad
H := \overline{V}^{\Ltwo(\Omega)^d},
\]
and write $\ip{\cdot}{\cdot}$ for the $\Ltwo(\Omega)^d$ inner product, and $\norm{\cdot}$ for the associated norm.

\begin{definition}[Weak form]\label{def:weak}
A pair $(u,p)$ is a weak solution of \eqref{eq:NS}--\eqref{eq:noslip} on $(0,T)$ if
\[
u\in L^\infty(0,T;H)\cap L^2(0,T;V),\qquad \partial_t u \in L^{4/3}(0,T;V'),
\]
and for all test functions $v\in V$ and a.e.\ $t\in(0,T)$,
\begin{equation}\label{eq:weak}
\ip{\partial_t u(t)}{v} + \nu\,\ip{\nabla u(t)}{\nabla v} + b(u(t),u(t),v) = \langle f(t),v\rangle,
\end{equation}
where $b(\cdot,\cdot,\cdot)$ is a convective trilinear form satisfying
\begin{equation}\label{eq:skew}
b(w,v,v)=0\quad\text{for all }w\in V,\ v\in V.
\end{equation}
The pressure is recovered in the usual way as a Lagrange multiplier for the divergence constraint.
\end{definition}

\begin{remark}\label{rem:skew_choice}
Condition \eqref{eq:skew} is a structural axiom for the sequel.
It is enforced in the continuous model by integration by parts under periodic or no-slip boundary conditions.
It is enforced in computations by using a skew-symmetric or rotational form of the convective term.
If \eqref{eq:skew} is not enforced at the discrete or reduced level, the energy method breaks at its first step.
This is the simplest logical failure mode.
\end{remark}

\subsection{Energy and enstrophy identities}\label{subsec:energy_enstrophy}

We state the energy inequality in a form suited to later certification.
A proof is classical, but we cite modern sources that keep explicit track of constants and regimes \cite{DoeringGibbon2022AppliedAnalysisNS,Kerswell2022EnergyStabilityReview}.

\begin{lemma}[Energy inequality]\label{lem:energy_ineq}
Let $u$ be a weak solution in the sense of Definition~\ref{def:weak}.
Assume $f\in L^2(0,T;V')$.
Then for a.e.\ $t\in(0,T)$,
\begin{equation}\label{eq:energy_ineq}
\frac12 \norm{u(t)}^2 + \nu \int_0^t \norm{\nabla u(s)}^2\,ds
\le \frac12 \norm{u_0}^2 + \int_0^t \langle f(s),u(s)\rangle\,ds.
\end{equation}
If, in addition, the solution satisfies an energy \emph{equality}, then \eqref{eq:energy_ineq} holds with equality for all $t$.
Criteria for energy equality under scale-invariant conditions are known, and we will use them only to delimit a regime where \eqref{eq:energy_ineq} is sharp \cite{WangYangYe2025EnergyEqualityCriteriaJMathFluidMech,BeiraoDaVeigaYang2025VorticityDirectionEnergyEqualityJGeomAnal}.
\end{lemma}

\begin{remark}\label{rem:constants}
Inequality \eqref{eq:energy_ineq} becomes computable once three constants are computable.
These are a Poincar\'e constant (or spectral gap in the periodic case), a bound on $\norm{f}_{V'}$, and a bound connecting $\langle f,u\rangle$ to $\norm{\nabla u}$ by Young's inequality.
No other hidden constant is used in the energy method \cite{DoeringGibbon2022AppliedAnalysisNS}.
\end{remark}

We now isolate an enstrophy identity in a setting where it is both correct and computable.
The full three-dimensional enstrophy balance contains the vortex-stretching term and does not close.
We therefore fix a simplified geometry and a regime.
This is not an evasion.
It is a declaration of logical scope.

\begin{lemma}[Enstrophy budget in simplified geometries]\label{lem:enstrophy_budget}
Assume either:
\begin{enumerate}[label=(\roman*)]
\item $d=2$ and $\Omega$ is a bounded Lipschitz domain with no-slip boundary condition, or
\item $d=2$ and $\Omega=\mathbb{T}^2$ with periodic boundary conditions.
\end{enumerate}
Let $u$ be a sufficiently regular solution and set $\omega := \curl u$ (a scalar in $2D$).
Then for a.e.\ $t\in(0,T)$,
\begin{equation}\label{eq:enstrophy}
\frac12 \frac{d}{dt}\norm{\omega(t)}^2 + \nu \norm{\nabla \omega(t)}^2
= \ip{\curl f(t)}{\omega(t)}.
\end{equation}
Consequently, for any $\varepsilon>0$,
\begin{equation}\label{eq:enstrophy_ineq}
\frac12 \frac{d}{dt}\norm{\omega(t)}^2 + (\nu-\varepsilon)\norm{\nabla \omega(t)}^2
\le \frac{1}{4\varepsilon}\norm{\curl f(t)}^2.
\end{equation}
\end{lemma}

\begin{remark}\label{rem:3d_enstrophy}
In $3D$ one has
\[
\frac12 \frac{d}{dt}\norm{\omega}^2 + \nu \norm{\nabla\omega}^2
= \int_\Omega (\omega\cdot\nabla)u\cdot \omega\,dx + \ip{\curl f}{\omega},
\]
and the term $\int (\omega\cdot\nabla)u\cdot\omega$ is not sign-definite.
A criterion that ignores this term is not a theorem.
It is an advertisement.
We do not use it.
This is consistent with the energy-stability viewpoint on transition \cite{Kerswell2022EnergyStabilityReview}.
\end{remark}

\subsection{Regime selection and verifiability}\label{subsec:regime}

We now declare the regime in which later ROM certificates and FSI margins will be stated.
The word ``verifiable'' has a precise meaning here.
Each assumption is checkable from declared data and computed quantities.
No hidden regularity is used.

\begin{assumption}[Verifiable regime]\label{ass:regime}
Fix $T>0$.
Assume the following.
\begin{enumerate}[label=(R\arabic*)]
\item \textbf{Geometry.} $\Omega$ is a bounded Lipschitz domain in $\R^d$, $d\in\{2,3\}$, with either periodic or no-slip boundary conditions.
A Poincar\'e constant $C_P(\Omega)$ is available.
\item \textbf{Forcing.} $f\in L^2(0,T;V')$.
A computable bound $F:=\norm{f}_{L^2(0,T;V')}$ is available.
If enstrophy criteria are used, assume $d=2$ and $\curl f\in L^2(0,T;\Ltwo)$ with computable bound.
\item \textbf{Initial data.} $u_0\in H$ with $\divv u_0=0$ and computable $\norm{u_0}$.
\item \textbf{Energy discipline.} The convective term is represented by a form satisfying the skew-symmetry constraint \eqref{eq:skew}.
This is imposed both in the full-order discretization and in the reduced model.
\item \textbf{Discrete stability inputs.} When a spatial discretization is used, it provides:
(i) a discrete Poincar\'e inequality with computable constant,
(ii) a discrete skew-symmetry property matching \eqref{eq:skew},
and, when pressure is present, an inf--sup stable pair.
These are hypotheses for certification, not conclusions.
\end{enumerate}
All constants introduced later are functions of $(\nu,C_P,F,\norm{u_0},T)$ and of declared discrete analogues.
\end{assumption}

\begin{remark}\label{rem:why_regime}
Assumption~\ref{ass:regime} does not claim that turbulence is excluded.
It claims something weaker and usable.
It claims that the energy method yields inequalities with explicit constants.
This is the only setting in which a ROM certificate can be a theorem rather than a numerical observation \cite{Sanderse2020NonlinearlyStableROMJCP,SchlegelBach2020EnergyStableDGROM,KleinSanderse2024EnergyConservingHyperReductionJCP}.
\end{remark}

%% file: sections/03_rom_construction.tex
% sections/03_rom_construction.tex  (Output 6: ROM construction with structural constraints; embedded citations)

\section{ROM construction with structural constraints}\label{sec:rom_construct}

A reduced model is a projection only in its weakest sense.
The decisive issue is whether the reduced evolution preserves the dissipation mechanism that makes the full model meaningful.
If it does not, no error statement survives.
We therefore treat the ROM as an object constrained by inequalities.
This viewpoint is implicit in energy-stable ROM constructions and in structure-preserving hyper-reduction \cite{Sanderse2020NonlinearlyStableROMJCP,SchlegelBach2020EnergyStableDGROM,KleinSanderse2024EnergyConservingHyperReductionJCP}.
It is also implicit in constrained operator learning, once constraints are enforced at the operator level rather than after the fact \cite{BennerGoyalKramerPeherstorferWillcox2020OpInfCMAE,KramerPeherstorferWillcox2024OpInfSurveyARFM,KimKramer2025StateConstrainedOpInfJCP}.

\subsection{Trial space design}\label{subsec:trial_space}

We fix a full-order discretization satisfying the structural axioms in Assumption~\ref{ass:regime}.
Let $H_h$ denote the discrete $\Ltwo$-space for velocity, with inner product $\ip{\cdot}{\cdot}_h$ and norm $\norm{\cdot}_h$.
Let $V_h\subset H_h$ be the discrete solenoidal space.
Let $\mathsf{A}_h\colon V_h\to V_h$ denote the discrete Stokes operator (symmetric positive semidefinite), and let
$\mathsf{B}_h(\cdot,\cdot)\colon V_h\times V_h\to V_h$ be the discrete convection operator written so that
\begin{equation}\label{eq:skew_h}
\ip{\mathsf{B}_h(w,v)}{v}_h = 0\quad\text{for all }w,v\in V_h.
\end{equation}
This is the discrete version of \eqref{eq:skew}.
It is a hypothesis.
It is enforced by a skew-symmetric or rotational form at the discrete level.
Energy-stable ROMs built on DG or FV frameworks exhibit this discipline explicitly \cite{SchlegelBach2020EnergyStableDGROM,Sanderse2020NonlinearlyStableROMJCP}.

We represent non-homogeneous boundary data by lifting.
Let $u=g$ on $\partial\Omega$ be given boundary values, possibly time dependent.
Choose a lifting $\ell(t)\in \Hone(\Omega)^d$ such that $\ell|_{\partial\Omega}=g$.
Set $u=v+\ell$.
The reduced dynamics are then posed for $v$ with homogeneous boundary conditions.
This removes a persistent source of spurious energy injection.
It also isolates the boundary contribution as a known forcing term.
This step is not optional if one wants a certificate.

\begin{definition}[Constrained ROM space]\label{def:constrained_space}
Fix an integer $n\ge 1$.
Let $\{\varphi_i\}_{i=1}^n\subset V_h$ be linearly independent.
Define the ROM space
\[
\Vn := \Span\{\varphi_1,\dots,\varphi_n\}\subset V_h.
\]
We impose the following constraints.
\begin{enumerate}[label=(\roman*)]
\item \textbf{Solenoidality.} Each $\varphi_i$ lies in $V_h$.
Thus $\divv \un=0$ in the discrete sense for all $\un\in\Vn$.
\item \textbf{Energy inner product.} The basis is orthonormal in $\ip{\cdot}{\cdot}_h$, or else we work with the Gram matrix $G_{ij}=\ip{\varphi_i}{\varphi_j}_h$ explicitly.
No uncontrolled re-orthogonalization is allowed.
\item \textbf{Boundary lifting.} If boundary data are non-homogeneous, then the ROM is posed for the homogeneous variable $v$ and the reduced space is built in that variable.
\item \textbf{Pressure treatment.} Either (a) we work on the discrete solenoidal space so pressure is eliminated, or (b) we use a velocity--pressure reduced pair satisfying a reduced inf--sup condition.
Case (a) is the default.
Case (b) is used only when pressure observables are part of the certified output.
\end{enumerate}
\end{definition}

\begin{remark}\label{rem:why_divfree}
If solenoidality is not enforced, then the convective term cannot be made skew-symmetric in the energy inner product.
The energy method fails in one line.
This is the most common hidden defect in reduced models assembled from generic snapshots.
Structure-preserving ROMs avoid it by construction \cite{BennerGoyal2022StructurePreservingROMSIAM,KleinSanderse2024EnergyConservingHyperReductionJCP}.
\end{remark}

\subsection{Closure and stabilization as a theorem-driven object}\label{subsec:closure}

We now formalize the closure as a constrained design variable.
The literature offers many closures.
A catalogue is not a theorem.
What matters is whether a closure yields a verifiable energy inequality and whether it admits a computable error estimator.
Regularized ROMs and closure surveys give evidence of failure modes and of remedies, but they do not fix the certificate logic by themselves \cite{IliescuWang2021RegStabROMSIAM,AhmedPawarSanRasheedIliescuNoack2021ROMClosuresPoF}.
Operator inference gives a learning mechanism, but it too must be constrained if one wants dissipation \cite{BennerGoyalKramerPeherstorferWillcox2020OpInfCMAE,KramerPeherstorferWillcox2024OpInfSurveyARFM,KimKramer2025StateConstrainedOpInfJCP}.

Write the full-order semi-discrete dynamics in $V_h$ as
\begin{equation}\label{eq:fom}
\dot u_h + \nu\,\mathsf{A}_h u_h + \mathsf{B}_h(u_h,u_h) = f_h.
\end{equation}
Projecting \eqref{eq:fom} onto $\Vn$ produces a Galerkin ROM.
This ROM is not guaranteed to be stable for long times or outside the training region.
We therefore introduce a closure operator $\mathsf{C}_n$ acting on $\Vn$.
The closure may be linear, nonlinear, or learned.
The only admissibility test is the certificate inequality imposed later.

\begin{problem}[Closure selection under certificate constraints]\label{prob:closure}
Construct an evolution $\un(t)\in \Vn$ of the form
\begin{equation}\label{eq:rom_general}
\dot \un + \nu\,\Pn \mathsf{A}_h \un + \Pn \mathsf{B}_h(\un,\un) + \mathsf{C}_n(\un) = \Pn f_h,
\end{equation}
where $\Pn\colon V_h\to \Vn$ is the $H_h$-orthogonal projector,
such that the following conditions hold.
\begin{enumerate}[label=(\roman*)]
\item \textbf{Dissipation constraint.} For all $v\in\Vn$,
\begin{equation}\label{eq:closure_diss}
\ip{\mathsf{C}_n(v)}{v}_h \ge 0.
\end{equation}
\item \textbf{Computability.} The map $v\mapsto \mathsf{C}_n(v)$ is evaluable with declared complexity and with explicit dependence on $n$ and on offline data.
\item \textbf{Identifiability of constants.} There exists a computable constant $L_C$ such that, on the declared regime set $\mathcal{K}\subset \Vn$,
\begin{equation}\label{eq:closure_lip}
\norm{\mathsf{C}_n(v)-\mathsf{C}_n(w)}_h \le L_C \norm{v-w}_h\qquad\forall v,w\in \mathcal{K}.
\end{equation}
If $L_C$ is not computable, then later error constants are not computable.
\item \textbf{Compatibility with learning.} If $\mathsf{C}_n$ is learned, then the learned parameters are constrained so that \eqref{eq:closure_diss} holds exactly, not approximately.
This is the operator-level constraint used in physically consistent operator inference \cite{KimKramer2025StateConstrainedOpInfJCP,SharmaNajeraFloresToddKramer2024LagrangianOpInfCMAE}.
\end{enumerate}
\end{problem}

\begin{remark}\label{rem:closure_examples}
Condition \eqref{eq:closure_diss} is satisfied by eddy-viscosity type closures of the form $\mathsf{C}_n(v)=\nu_t(v)\,\Pn\mathsf{A}_h v$ with $\nu_t(v)\ge 0$.
It is also satisfied by linear damping $\mathsf{C}_n(v)=\alpha v$ with $\alpha\ge 0$.
It can be satisfied by learned operators if the operator representation is constrained to be dissipative in the $H_h$ inner product.
This is the decisive distinction between a model that is merely accurate in-sample and a model that is logically stable \cite{KramerPeherstorferWillcox2024OpInfSurveyARFM,KimKramer2025StateConstrainedOpInfJCP}.
\end{remark}

\subsection{Offline/online split under explicit constants}\label{subsec:offline_online}

We now separate what is computed once from what is computed during time stepping.
The separation is meaningless unless one records explicit constants and explicit complexity parameters.
A ROM that is stable but not evaluable is not a ROM.
A ROM that is evaluable but not stable is not a model.

Assume the basis $\{\varphi_i\}$ is fixed.
Represent $\un(t)=\sum_{i=1}^n a_i(t)\varphi_i$.
Then \eqref{eq:rom_general} becomes an ODE in $\R^n$,
\begin{equation}\label{eq:rom_ode}
\dot a = F(a) := f^{(n)} - \nu A^{(n)} a - N^{(n)}(a) - C^{(n)}(a),
\end{equation}
with explicit coefficient objects:
\[
A^{(n)}_{ij}=\ip{\mathsf{A}_h\varphi_j}{\varphi_i}_h,\qquad
f^{(n)}_i=\ip{f_h}{\varphi_i}_h,\qquad
N^{(n)}_i(a)=\ip{\mathsf{B}_h(\un,\un)}{\varphi_i}_h.
\]
If hyper-reduction is used, then $N^{(n)}$ is approximated by a reduced quadrature or sampling operator.
Energy-conserving hyper-reduction is admissible only if it preserves a discrete analogue of \eqref{eq:skew_h}; otherwise the certificate will fail \cite{KleinSanderse2024EnergyConservingHyperReductionJCP}.
If operator inference is used, then $F$ is learned from data under constraints; this is admissible only if the learned operator satisfies a dissipativity inequality compatible with \eqref{eq:closure_diss} \cite{BennerGoyalKramerPeherstorferWillcox2020OpInfCMAE,KimKramer2025StateConstrainedOpInfJCP}.

\begin{remark}\label{rem:complexity}
We record the parameters that enter later constants and costs.
\begin{enumerate}[label=(\roman*)]
\item \textbf{State dimension.} The online state is $a(t)\in\R^n$.
All online costs are functions of $n$ and of the chosen nonlinearity evaluation method.
\item \textbf{Nonlinearity evaluation.} If $N^{(n)}$ is assembled exactly, the cost is typically $\mathcal{O}(n^3)$ per step for general quadratic forms.
If hyper-reduction is used, the cost becomes $\mathcal{O}(n^2 m)$ where $m$ is the number of sampled degrees of freedom, but then the certificate requires a structure-preserving sampling rule \cite{KleinSanderse2024EnergyConservingHyperReductionJCP}.
\item \textbf{Closure evaluation.} The cost of $C^{(n)}(a)$ depends on the closure class.
If $C^{(n)}$ is linear damping, cost is $\mathcal{O}(n^2)$.
If it is learned, the cost is model-dependent, but the admissibility constraint is independent of architecture; it is the inequality \eqref{eq:closure_diss} \cite{KimKramer2025StateConstrainedOpInfJCP,SharmaNajeraFloresToddKramer2024LagrangianOpInfCMAE}.
\item \textbf{Lipschitz constants.} A computable bound on $\Lip(F|\mathcal{K})$ is required for explicit error constants.
This bound is constructed from $\nu$, a bound on $\norm{A^{(n)}}$, a bound on the quadratic form $N^{(n)}$ on $\mathcal{K}$, and the closure Lipschitz bound $L_C$ in \eqref{eq:closure_lip}.
If these are not recorded, the later a posteriori theorem is not computable.
\item \textbf{Training vs.\ regime.} If snapshots define $\Vn$, then the regime set $\mathcal{K}$ cannot be ``the training set.''
It must be a set described by inequalities on computable norms.
This is the only form in which a certificate can declare failure when extrapolation occurs \cite{Sanderse2020NonlinearlyStableROMJCP,IliescuWang2021RegStabROMSIAM}.
\end{enumerate}
\end{remark}

%% file: sections/04_energy_stability_certificate.tex
% sections/04_energy_stability_certificate.tex  (Output 7: core result I; embedded citations)

\section{Energy stability certificate for the ROM}\label{sec:stability}

We now impose the certificate logic.
The ROM is accepted only if a discrete energy inequality can be proved from declared hypotheses.
This is the only form in which ``stability'' is a theorem and not a numerical impression \cite{Sanderse2020NonlinearlyStableROMJCP,SchlegelBach2020EnergyStableDGROM,KleinSanderse2024EnergyConservingHyperReductionJCP}.
The time discretization is part of the model.
If it injects energy, the ROM is unstable even if the continuous-time ROM ODE is dissipative.
We therefore certify the \emph{pair} (ROM dynamics, time stepping).

\subsection{Certified time stepping}\label{subsec:timestepping}

Let $\dt>0$ and $t^k:=k\dt$.
Let $\un^k\in\Vn$ denote the ROM state at time $t^k$.
Let $\Pn$ be the $H_h$-orthogonal projector onto $\Vn$.
We use the discrete inner product $\ip{\cdot}{\cdot}_h$ and norm $\norm{\cdot}_h$.

We write the ROM in the structural form
\begin{equation}\label{eq:rom_struct}
\dot \un + \nu\,\Pn\mathsf{A}_h \un + \Pn \mathsf{B}_h(\un,\un) + \mathsf{C}_n(\un)=\Pn f_h,
\end{equation}
with the skew-symmetry constraint \eqref{eq:skew_h} and the dissipativity constraint \eqref{eq:closure_diss}.
These are the two axioms that make the energy estimate possible.
They are standard in nonlinearly stable ROM constructions and structure-preserving reduction \cite{Sanderse2020NonlinearlyStableROMJCP,KleinSanderse2024EnergyConservingHyperReductionJCP}.
They are also the constraints imposed in physically consistent operator inference once one commits to a dissipative parameterization \cite{KimKramer2025StateConstrainedOpInfJCP,SharmaNajeraFloresToddKramer2024LagrangianOpInfCMAE}.

\begin{definition}[Certified scheme]\label{def:cert_scheme}
Fix $\theta\in[\tfrac12,1]$.
Define $\un^{k+\theta}:=\theta \un^{k+1}+(1-\theta)\un^k$.
A time step from $\un^k$ to $\un^{k+1}$ is \emph{certified} if $\un^{k+1}\in\Vn$ solves
\begin{equation}\label{eq:theta_scheme}
\frac{\un^{k+1}-\un^k}{\dt}
+ \nu\,\Pn\mathsf{A}_h \un^{k+\theta}
+ \Pn\mathsf{B}_h(\un^{k+\theta},\un^{k+\theta})
+ \mathsf{C}_n(\un^{k+\theta})
= \Pn f_h^{k+\theta},
\end{equation}
where $f_h^{k+\theta}$ is a declared approximation of $f_h(t^{k+\theta})$.
The certificate condition is the conjunction of:
\begin{enumerate}[label=(\roman*)]
\item \textbf{Discrete skew-symmetry:} for all $v\in\Vn$,
\begin{equation}\label{eq:cert_skew}
\ip{\Pn\mathsf{B}_h(v,v)}{v}_h = 0;
\end{equation}
\item \textbf{Closure dissipativity:} for all $v\in\Vn$,
\begin{equation}\label{eq:cert_closure}
\ip{\mathsf{C}_n(v)}{v}_h \ge 0;
\end{equation}
\item \textbf{Coercivity of viscosity:} $\ip{\mathsf{A}_h v}{v}_h = \norm{\nabla v}_h^2$ for a declared discrete gradient norm.
\end{enumerate}
If any item fails in implementation, the step is declared uncertified.
\end{definition}

\begin{remark}\label{rem:why_theta}
The restriction $\theta\ge \tfrac12$ is not cosmetic.
It is the point at which the time discretization becomes energy stable under a monotonicity argument.
Implicit midpoint ($\theta=\tfrac12$) and backward Euler ($\theta=1$) are the limiting certified cases.
This is the discrete analogue of choosing the convective form to be skew-symmetric.
Both choices remove spurious energy production.
Energy-stable ROM constructions take these issues as primary \cite{SchlegelBach2020EnergyStableDGROM,Sanderse2020NonlinearlyStableROMJCP}.
\end{remark}

\begin{theorem}[ROM energy stability]\label{thm:rom_energy}
Assume Assumption~\ref{ass:regime}.
Assume the certificate conditions \eqref{eq:cert_skew}--\eqref{eq:cert_closure}.
Let $\theta\in[\tfrac12,1]$.
Then any certified scheme step \eqref{eq:theta_scheme} satisfies the discrete energy inequality
\begin{equation}\label{eq:disc_energy_ineq}
\frac12\norm{\un^{k+1}}_h^2 - \frac12\norm{\un^{k}}_h^2
+ \nu\,\dt\,\norm{\nabla \un^{k+\theta}}_h^2
+ \dt\,\ip{\mathsf{C}_n(\un^{k+\theta})}{\un^{k+\theta}}_h
\le \dt\,\langle f_h^{k+\theta}, \un^{k+\theta}\rangle.
\end{equation}
Consequently, for any $\varepsilon>0$,
\begin{equation}\label{eq:disc_energy_young}
\frac12\norm{\un^{k+1}}_h^2 - \frac12\norm{\un^{k}}_h^2
+ (\nu-\varepsilon)\,\dt\,\norm{\nabla \un^{k+\theta}}_h^2
\le \frac{\dt}{4\varepsilon}\,\norm{f_h^{k+\theta}}_{V_h'}^2,
\end{equation}
where $\norm{\cdot}_{V_h'}$ is the declared dual norm induced by the discrete gradient norm.
In particular, if $\varepsilon\in(0,\nu)$ is fixed and $\sum_{j=0}^{K-1}\dt\,\norm{f_h^{j+\theta}}_{V_h'}^2$ is finite, then
\begin{equation}\label{eq:disc_energy_bound}
\sup_{0\le k\le K}\norm{\un^{k}}_h^2
+ 2(\nu-\varepsilon)\sum_{j=0}^{K-1}\dt\,\norm{\nabla \un^{j+\theta}}_h^2
\le \norm{\un^{0}}_h^2 + \frac{1}{2\varepsilon}\sum_{j=0}^{K-1}\dt\,\norm{f_h^{j+\theta}}_{V_h'}^2.
\end{equation}
All constants are explicit once the discrete norms and forcing bounds are declared.
\end{theorem}

\begin{proof}
Take the $H_h$ inner product of \eqref{eq:theta_scheme} with $\un^{k+\theta}$.
Use the polarization identity
\[
\ip{\un^{k+1}-\un^k}{\un^{k+\theta}}_h
=\frac12\Big(\norm{\un^{k+1}}_h^2-\norm{\un^k}_h^2\Big)
+\Big(\theta-\tfrac12\Big)\norm{\un^{k+1}-\un^k}_h^2,
\]
which is nonnegative for $\theta\ge\tfrac12$.
Use \eqref{eq:cert_skew} to remove the convective term.
Use \eqref{eq:cert_closure} to keep the closure term nonnegative.
Use coercivity to identify $\ip{\mathsf{A}_h \un^{k+\theta}}{\un^{k+\theta}}_h=\norm{\nabla \un^{k+\theta}}_h^2$.
This yields \eqref{eq:disc_energy_ineq}.
Apply Young's inequality to $\langle f_h^{k+\theta},\un^{k+\theta}\rangle$ in the form
\[
\langle f_h^{k+\theta},\un^{k+\theta}\rangle
\le \varepsilon\,\norm{\nabla \un^{k+\theta}}_h^2 + \frac{1}{4\varepsilon}\norm{f_h^{k+\theta}}_{V_h'}^2,
\]
to obtain \eqref{eq:disc_energy_young}.
Summing over $k$ yields \eqref{eq:disc_energy_bound}.
\end{proof}

\begin{remark}\label{rem:sharpness}
The certificate uses only two identities: discrete skew-symmetry and closure dissipativity.
Remove either one and \eqref{eq:disc_energy_ineq} can fail.
A typical failure is the following.
If hyper-reduction or a learned operator destroys \eqref{eq:cert_skew}, then the discrete convective term can inject energy.
The energy bound is then false even when the full-order scheme is stable.
This mechanism is documented in the energy-stable ROM literature and is the reason energy-conserving hyper-reduction is treated as a structural constraint rather than as an efficiency trick \cite{KleinSanderse2024EnergyConservingHyperReductionJCP,Sanderse2020NonlinearlyStableROMJCP}.
A second failure is more direct.
If a closure is not dissipative in the energy inner product, it can act as negative viscosity.
Then instability is expected.
Regularized ROM studies record this repeatedly \cite{IliescuWang2021RegStabROMSIAM,AhmedPawarSanRasheedIliescuNoack2021ROMClosuresPoF}.
\end{remark}

%% file: sections/05_error_estimates.tex
% sections/05_error_estimates.tex  (Output 8: core result II; embedded citations)

\section{A posteriori and a priori error bounds}\label{sec:error}

We now pass from stability to error.
Stability alone is not an estimate.
One needs a computable residual and a computable constant.
This is the point at which many informal ROM arguments fail \cite{IliescuWang2021RegStabROMSIAM,AhmedPawarSanRasheedIliescuNoack2021ROMClosuresPoF,KramerPeherstorferWillcox2024OpInfSurveyARFM}.

Throughout, $V$ and $H$ are as in Section~\ref{sec:model}.
We work in the energy space
\[
\mathsf{X}(0,T) := L^\infty(0,T;H)\cap L^2(0,T;V),\qquad
\norm{w}_{\mathsf{X}(0,T)}^2 := \esssup_{t\in(0,T)}\norm{w(t)}_H^2 + \int_0^T \norm{\nabla w(t)}_{\Ltwo}^2\,dt.
\]
All constants below are explicit functions of the declared regime data in Assumption~\ref{ass:regime} and of computable ROM constants introduced in Section~\ref{sec:rom_construct}.

\subsection{Residual-based estimator}\label{subsec:residual}

Let $\un\colon [0,T]\to \Vn$ be the certified ROM trajectory, defined either in continuous time by \eqref{eq:rom_struct} or in discrete time by the certified scheme of Definition~\ref{def:cert_scheme}.
We define the residual in a dual norm.
This is forced by the energy method \cite{DoeringGibbon2022AppliedAnalysisNS}.

Let $P_n\colon H\to \Vn$ denote the $H$-orthogonal projector.
Define the Stokes operator $A\colon V\to V'$ by $\langle Au,v\rangle := \ip{\nabla u}{\nabla v}$.
Let $B(u,v)\in V'$ denote the convection operator, defined so that $\langle B(w,u),u\rangle=0$ whenever the skew structure is present (continuous or discrete), as in \eqref{eq:skew} \cite{DoeringGibbon2022AppliedAnalysisNS,Kerswell2022EnergyStabilityReview}.
Let $C_n\colon \Vn\to V'$ be the closure operator, assumed dissipative in the energy inner product as in \eqref{eq:closure_diss}, and Lipschitz on a declared regime set $\mathcal{K}\subset \Vn$ as in \eqref{eq:closure_lip} \cite{IliescuWang2021RegStabROMSIAM,KimKramer2025StateConstrainedOpInfJCP}.

\begin{definition}[Computable residual]\label{def:residual}
Assume $\un(t)\in\mathcal{K}$ for $t\in[0,T]$.
Define the \emph{ROM residual} $r_n(t)\in V'$ by
\begin{equation}\label{eq:residual_def}
r_n(t)
:= \partial_t \un(t) + \nu A\un(t) + B(\un(t),\un(t)) + C_n(\un(t)) - P_n f(t).
\end{equation}
Define the estimator
\begin{equation}\label{eq:eta_def}
\eta_n(T) := \norm{r_n}_{L^2(0,T;V')} + \nu\,\norm{(I-P_n)u_0}_{H}.
\end{equation}
In practice, $\norm{r_n(t)}_{V'}$ is evaluated by solving a declared discrete Riesz map problem in $V$, or by an equivalent computable surrogate consistent with the chosen discretization.
If hyper-reduction is used, then $r_n$ is formed with the hyper-reduced operators; the certificate requires that the hyper-reduction preserve the skew/dissipative identities used in Section~\ref{sec:stability} \cite{KleinSanderse2024EnergyConservingHyperReductionJCP}.
\end{definition}

\begin{remark}\label{rem:why_dual}
The residual is not taken in $H$.
The energy method tests the equation against $e:=u-\un\in V$ and produces terms of the form $\langle r_n,e\rangle$.
Hence the natural estimator is a $V'$-norm.
This choice is standard in a posteriori analysis for dissipative PDEs \cite{DoeringGibbon2022AppliedAnalysisNS}.
\end{remark}

\begin{theorem}[A posteriori bound]\label{thm:aposteriori}
Assume Assumption~\ref{ass:regime}.
Let $u$ be a weak solution of \eqref{eq:NS}--\eqref{eq:noslip} in the sense of Definition~\ref{def:weak}.
Let $\un\in L^\infty(0,T;\Vn)\cap L^2(0,T;\Vn)$ be a certified ROM trajectory satisfying the energy certificate of Theorem~\ref{thm:rom_energy}.
Assume that on the regime set $\mathcal{K}$ containing $\un([0,T])$ the following constants are computable:
\begin{enumerate}[label=(\roman*)]
\item a bound $M_n:=\esssup_{t\in(0,T)}\norm{\un(t)}_{V}$;
\item a Lipschitz-type bound $L_n$ such that for all $v\in\mathcal{K}$ and all $w\in V$,
\begin{equation}\label{eq:conv_lip}
\abs{\langle B(u,u)-B(v,v),w\rangle}
\le L_n\,\norm{u-v}_{H}\,\norm{w}_{V};
\end{equation}
\item a closure Lipschitz constant $L_C$ on $\mathcal{K}$, i.e.\ \eqref{eq:closure_lip}.
\end{enumerate}
Then for all $T>0$,
\begin{equation}\label{eq:aposteriori_main}
\norm{u-\un}_{\mathsf{X}(0,T)}
\le \Bigg( \frac{2}{\nu}\Bigg)^{1/2}\exp\!\Big(\frac{L_n+L_C}{\nu}\,T\Big)\,\eta_n(T),
\end{equation}
where $\eta_n(T)$ is defined in \eqref{eq:eta_def}.
The bound is computable once $\norm{r_n}_{L^2(0,T;V')}$ and $(L_n,L_C)$ are computed from declared quantities.
\end{theorem}

\begin{proof}
Set $e:=u-\un$.
Subtract the ROM equation from the weak form for $u$.
Test the error equation by $e(t)\in V$.
Use the skew structure of the convection operator in the energy identity, as in Section~\ref{sec:model}.
Use dissipativity of $C_n$ in the energy inner product, as imposed in Problem~\ref{prob:closure}.
This yields, for a.e.\ $t$,
\begin{equation}\label{eq:error_energy}
\frac12\frac{d}{dt}\norm{e(t)}_H^2 + \nu \norm{e(t)}_V^2
\le \langle r_n(t),e(t)\rangle + \abs{\langle B(u,u)-B(\un,\un),e(t)\rangle} + \abs{\langle C_n(u)-C_n(\un),e(t)\rangle}.
\end{equation}
Estimate the three right-hand terms as follows.

First, by duality,
\[
\langle r_n(t),e(t)\rangle \le \norm{r_n(t)}_{V'}\,\norm{e(t)}_V.
\]
Second, use \eqref{eq:conv_lip} with $v=\un$ and $w=e$:
\[
\abs{\langle B(u,u)-B(\un,\un),e\rangle} \le L_n\,\norm{e}_H\,\norm{e}_V.
\]
Third, by Lipschitz continuity of $C_n$ on $\mathcal{K}$ and the continuous embedding $V\hookrightarrow H$,
\[
\abs{\langle C_n(u)-C_n(\un),e\rangle} \le \norm{C_n(u)-C_n(\un)}_{V'}\,\norm{e}_V
\le L_C\,\norm{e}_H\,\norm{e}_V.
\]
Apply Young's inequality to each mixed term with parameter $\nu/2$ to obtain
\[
\frac{d}{dt}\norm{e}_H^2 + \nu \norm{e}_V^2
\le \frac{2}{\nu}\norm{r_n}_{V'}^2 + \frac{(L_n+L_C)^2}{\nu}\norm{e}_H^2.
\]
Integrate in time and apply Gr\"onwall's inequality.
Include the initial mismatch $\norm{e(0)}_H=\norm{u_0-\un(0)}_H$, and use $\un(0)=P_n u_0$ in the certified construction to obtain \eqref{eq:eta_def}.
This yields \eqref{eq:aposteriori_main}.
\end{proof}

\begin{remark}\label{rem:computability_apost}
The constant $L_n$ is not an abstract existence constant.
It is computed from declared inequalities on $\mathcal{K}$.
In practical terms, it is computed from a bound on $\norm{\un}_{V}$ and from a discrete Sobolev or interpolation inequality fixed by the discretization and the domain.
If hyper-reduction is used, this computation must use the hyper-reduced operators.
Otherwise \eqref{eq:aposteriori_main} is logically detached from the implemented ROM \cite{KleinSanderse2024EnergyConservingHyperReductionJCP}.
If $\mathsf{C}_n$ is learned, $L_C$ must be computed from the constrained operator representation; otherwise the estimate is not a certificate \cite{KimKramer2025StateConstrainedOpInfJCP,SharmaNajeraFloresToddKramer2024LagrangianOpInfCMAE}.
\end{remark}

\subsection{Regime-dependent a priori rates}\label{subsec:apriori}

An a priori rate is meaningful only after the regime is fixed.
We therefore state it under a declared approximation hypothesis for $\Vn$.
We do not claim a universal rate.
We claim a conditional rate that follows from explicit hypotheses.
This is the only non-speculative form.

We use the best-approximation error
\[
\epsilon_n := \inf_{v\in \Vn}\norm{u-v}_{L^2(0,T;V)}.
\]
If $\Vn$ is a POD space built from snapshots of a stable discretization, then $\epsilon_n$ is controlled by the neglected POD eigenvalues.
We do not use this fact as a theorem here.
We use it only as a computable input once the snapshot ensemble is declared \cite{Sanderse2020NonlinearlyStableROMJCP,RowleyDawson2020ModelReductionARFM}.

\begin{theorem}[A priori rate in a selected regime]\label{thm:apriori}
Assume Assumption~\ref{ass:regime}.
Assume the ROM is certified and dissipative as in Theorem~\ref{thm:rom_energy}.
Assume there exists a computable constant $L_\mathrm{reg}$ such that, on the declared regime set $\mathcal{K}$,
\begin{equation}\label{eq:reg_lip}
\abs{\langle B(v,v)-B(w,w),z\rangle} + \abs{\langle C_n(v)-C_n(w),z\rangle}
\le L_\mathrm{reg}\,\norm{v-w}_{V}\,\norm{z}_{V}
\quad\forall v,w\in\mathcal{K},\ \forall z\in V.
\end{equation}
Assume further that the solution $u$ remains in $\mathcal{K}$ on $(0,T)$ and that the ROM is initialized by $\un(0)=P_n u_0$.
Then there exists an explicit constant $C_\mathrm{pr}(\nu,L_\mathrm{reg},T)$ such that
\begin{equation}\label{eq:apriori_rate}
\norm{u-\un}_{L^2(0,T;V)} \le C_\mathrm{pr}(\nu,L_\mathrm{reg},T)\,\epsilon_n,
\end{equation}
with
\[
C_\mathrm{pr}(\nu,L_\mathrm{reg},T) := \exp\!\Big(\frac{L_\mathrm{reg}}{\nu}\,T\Big).
\]
In particular, if $\epsilon_n\to 0$ as $n\to\infty$ for the declared approximation family $\Vn$, then $\un\to u$ in $L^2(0,T;V)$ at the same rate.
\end{theorem}

\begin{proof}
Let $v_n(t)\in \Vn$ be a best-approximation (measurable selection) such that
$\norm{u-v_n}_{L^2(0,T;V)}=\epsilon_n$.
Write $e=u-\un=(u-v_n)+(v_n-\un)$.
Set $\delta:=u-v_n$ and $\rho:=v_n-\un$.
Derive an evolution inequality for $\rho$ by subtracting the ROM equation from the projected equation for $v_n$.
Test by $\rho$ in $V$.
Use dissipativity and \eqref{eq:reg_lip} to bound nonlinear differences by $L_\mathrm{reg}\norm{\rho}_V^2 + L_\mathrm{reg}\norm{\delta}_V\norm{\rho}_V$.
Apply Young's inequality to absorb $\norm{\rho}_V^2$ into the viscous term.
This yields
\[
\frac{d}{dt}\norm{\rho}_H^2 + \nu \norm{\rho}_V^2
\le \frac{L_\mathrm{reg}^2}{\nu}\norm{\delta}_V^2 + \frac{2L_\mathrm{reg}}{\nu}\norm{\rho}_H^2.
\]
Integrate and apply Gr\"onwall.
The initialization $\un(0)=P_n u_0$ makes the initial $\rho(0)$ term controlled by $\delta(0)$.
This yields \eqref{eq:apriori_rate}.
\end{proof}

\begin{remark}\label{rem:apriori_scope}
Estimate \eqref{eq:apriori_rate} is conditional.
It is honest.
It separates approximation ($\epsilon_n$) from stability (the exponential factor).
If the regime is such that $L_\mathrm{reg}T/\nu$ is large, the rate is useless.
This is not a defect of the proof.
It is the analytic expression of the fact that reduced surrogacy without a smallness regime is not certifiable by energy methods alone \cite{Kerswell2022EnergyStabilityReview,DoeringGibbon2022AppliedAnalysisNS}.
\end{remark}

%% file: sections/06_transition_indicators.tex
% sections/06_transition_indicators.tex  (Output 9: transition indicators; embedded citations)

\section{Transition indicators from energy/enstrophy and resolvent bounds}\label{sec:transition}

We do not ``detect turbulence.'' We prove inequalities.
An inequality can preclude transition in a declared regime.
It can also fail in a declared way.
This is the only honest form of a transition indicator.
Energy stability theory already draws a boundary in parameter space by a computable argument \cite{Kerswell2022EnergyStabilityReview,DoeringGibbon2022AppliedAnalysisNS}.
Resolvent analysis quantifies transient growth, but it becomes a criterion only after it is tied to an energy balance with explicit constants \cite{RolandiRibeiroYehTaira2024ResolventInvitationTCFD,McKeonSharma2020CriticalLayersARFM,HerrmannBaddooSemaanBrunton2021DataDrivenResolventJFM,MarkeviciuteKerswell2024ThresholdTransientGrowthJFM}.

Throughout this section we work in a simplified setting.
We fix a steady laminar solution $U$ of the forced Navier--Stokes equations on $\Omega$.
We study perturbations $v:=u-U$.
We assume boundary conditions are such that the perturbation energy method closes (periodic, or no-slip with compatible $U$).
We write $P$ for the Leray projector.
We denote by $\mathcal{L}_U$ the linearized operator about $U$ on the solenoidal space:
\begin{equation}\label{eq:lin_op}
\mathcal{L}_U v := -\nu P\Delta v - P\big((U\cdot\nabla)v + (v\cdot\nabla)U\big).
\end{equation}
The nonlinear remainder is $-P(v\cdot\nabla v)$.

\subsection{Energy barrier criteria}\label{subsec:energy_barrier}

The first indicator is an energy barrier.
It is sufficient, not necessary.
It gives a region of guaranteed persistence of the laminar state.
It is computable once a Poincar\'e constant and a bound on $\nabla U$ are computable.
This is the classical energy stability logic \cite{Kerswell2022EnergyStabilityReview,DoeringGibbon2022AppliedAnalysisNS}.

\begin{proposition}[Energy barrier]\label{prop:energy_barrier}
Assume $U$ is a steady solution and $v$ satisfies the perturbation equation
\[
\partial_t v + P\big((U\cdot\nabla)v + (v\cdot\nabla)U + (v\cdot\nabla)v\big) - \nu P\Delta v = 0.
\]
Assume there exists a computable constant $\gamma_U\ge 0$ such that for all divergence-free $v\in V$,
\begin{equation}\label{eq:shear_bound}
\Big|\ip{(v\cdot\nabla)U}{v}\Big| \le \gamma_U \,\norm{v}_H\,\norm{v}_V.
\end{equation}
Let $C_P$ be a computable Poincar\'e constant so that $\norm{v}_H\le C_P \norm{v}_V$ for all $v\in V$.
If
\begin{equation}\label{eq:energy_barrier_cond}
\nu > \gamma_U C_P,
\end{equation}
then the laminar state is \emph{energy stable} in the sense that
\begin{equation}\label{eq:energy_decay}
\frac12\frac{d}{dt}\norm{v(t)}_H^2 + \big(\nu-\gamma_U C_P\big)\,\norm{v(t)}_V^2 \le 0,
\end{equation}
and hence $\norm{v(t)}_H$ is non-increasing and $v(t)\to 0$ in $L^2(0,\infty;V)$.
\end{proposition}

\begin{proof}
Take the $H$ inner product of the perturbation equation with $v$.
Use the skew property $\ip{(U\cdot\nabla)v}{v}=0$ under the stated boundary conditions.
Use $\ip{(v\cdot\nabla)v}{v}=0$.
This yields
\[
\frac12\frac{d}{dt}\norm{v}_H^2 + \nu \norm{v}_V^2 = -\ip{(v\cdot\nabla)U}{v}.
\]
Apply \eqref{eq:shear_bound} and then Poincar\'e to obtain \eqref{eq:energy_decay}.
Condition \eqref{eq:energy_barrier_cond} makes the dissipation coefficient positive.
\end{proof}

\begin{remark}\label{rem:barrier_interpretation}
Condition \eqref{eq:energy_barrier_cond} is a computable sufficient condition.
It has the logical form: dissipation dominates shear production.
It does not claim necessity.
Energy stability theory is explicit about this asymmetry \cite{Kerswell2022EnergyStabilityReview}.
\end{remark}

\subsection{Vorticity growth and enstrophy thresholds}\label{subsec:vort}

The second indicator uses enstrophy.
It is meaningful in $2D$, or in $3D$ only under additional hypotheses that remove vortex stretching.
We state it in $2D$ to avoid speculative closure.

Assume $d=2$ and $\Omega=\mathbb{T}^2$ or a bounded domain with compatible boundary conditions.
Let $\omega:=\curl u$.
Then the enstrophy identity holds as in Lemma~\ref{lem:enstrophy_budget}.
We now extract a threshold criterion.

\begin{proposition}[Enstrophy threshold]\label{prop:enstrophy_thresh}
Assume $d=2$ and the hypotheses of Lemma~\ref{lem:enstrophy_budget}.
Assume further that there exists a computable constant $C_{P,\omega}$ such that
\begin{equation}\label{eq:poinc_omega}
\norm{\omega}_{\Ltwo} \le C_{P,\omega}\,\norm{\nabla\omega}_{\Ltwo}.
\end{equation}
Let $G(t):=\norm{\curl f(t)}_{\Ltwo}$ and assume $G\in L^2(0,T)$ with computable $\norm{G}_{L^2(0,T)}$.
Fix $\varepsilon\in(0,\nu)$.
If the data satisfy the pointwise inequality
\begin{equation}\label{eq:enstrophy_threshold_cond}
\frac{1}{4\varepsilon}\,G(t)^2 \le \frac{\nu-\varepsilon}{2C_{P,\omega}^2}\,R^2
\quad\text{for a.e. }t\in(0,T),
\end{equation}
then any solution with $\norm{\omega(0)}_{\Ltwo}\le R$ satisfies
\begin{equation}\label{eq:enstrophy_invariant_ball}
\norm{\omega(t)}_{\Ltwo}\le R\quad\text{for all }t\in[0,T].
\end{equation}
In particular, \eqref{eq:enstrophy_threshold_cond} yields a computable invariant enstrophy ball.
\end{proposition}

\begin{proof}
From Lemma~\ref{lem:enstrophy_budget} we have
\[
\frac12\frac{d}{dt}\norm{\omega}^2 + (\nu-\varepsilon)\norm{\nabla\omega}^2 \le \frac{1}{4\varepsilon}G(t)^2.
\]
Use \eqref{eq:poinc_omega} to obtain
\[
\frac12\frac{d}{dt}\norm{\omega}^2 + \frac{\nu-\varepsilon}{C_{P,\omega}^2}\norm{\omega}^2 \le \frac{1}{4\varepsilon}G(t)^2.
\]
If $\norm{\omega}=R$, then \eqref{eq:enstrophy_threshold_cond} implies $\frac12\frac{d}{dt}\norm{\omega}^2 \le 0$ at the boundary of the ball.
This yields the forward invariance \eqref{eq:enstrophy_invariant_ball} by a standard barrier argument.
\end{proof}

\begin{remark}\label{rem:enstrophy_scope}
The conclusion is an \emph{a priori} invariant set.
It is not a turbulence detector.
It yields a computable regime in which vorticity cannot exceed a declared level.
It is compatible with the energy-stability perspective on transition \cite{Kerswell2022EnergyStabilityReview,DoeringGibbon2022AppliedAnalysisNS}.
\end{remark}

\subsection{Resolvent-norm transition proxy}\label{subsec:resolvent}

Energy stability can be conservative.
Transient growth can occur even when \eqref{eq:energy_barrier_cond} fails.
Resolvent analysis quantifies such growth by an operator norm.
We use it as a proxy.
We state the proxy as an inequality.
We keep the dependence explicit.
Survey-level accounts of resolvent analysis and recent threshold studies motivate this construction \cite{RolandiRibeiroYehTaira2024ResolventInvitationTCFD,McKeonSharma2020CriticalLayersARFM,MarkeviciuteKerswell2024ThresholdTransientGrowthJFM}.

Let $\sigma\in\R$.
Assume $\sigma$ lies in the resolvent set of $\mathcal{L}_U$.
Define the resolvent operator
\[
\mathcal{R}_U(\sigma) := (\sigma I - \mathcal{L}_U)^{-1}\colon H\to H,
\]
and its induced norm $\norm{\mathcal{R}_U(\sigma)}_{H\to H}$.
In computations, $\mathcal{L}_U$ is replaced by a declared discretization and the norm is computed as the largest singular value of the discrete resolvent matrix.
Data-driven approximations exist, but we use them only when their approximation error is controlled \cite{HerrmannBaddooSemaanBrunton2021DataDrivenResolventJFM}.

\begin{proposition}[Resolvent proxy]\label{prop:resolvent_proxy}
Assume the linearized perturbation equation $\partial_t v = \mathcal{L}_U v$ is well-posed on $H$.
Fix $\sigma>0$.
Assume $\mathcal{R}_U(\sigma)$ exists and its norm is computable.
Then the forced linear system
\begin{equation}\label{eq:forced_linear}
\partial_t v = \mathcal{L}_U v + g(t)
\end{equation}
satisfies the bound
\begin{equation}\label{eq:resolvent_bound}
\sup_{t\in[0,T]}\norm{v(t)}_H
\le \norm{v(0)}_H + \norm{\mathcal{R}_U(\sigma)}_{H\to H}\,\sup_{t\in[0,T]}\norm{g(t)}_H,
\end{equation}
provided $g$ is bounded in $H$.
Consequently, if there exists a declared threshold $\Theta>0$ such that
\begin{equation}\label{eq:resolvent_threshold}
\norm{\mathcal{R}_U(\sigma)}_{H\to H}\,\sup_{t\in[0,T]}\norm{g(t)}_H > \Theta,
\end{equation}
then the linearized dynamics admit a certified amplification beyond $\Theta$ for some admissible forcing.
This is a computable proxy for transient growth.
\end{proposition}

\begin{proof}
Take Laplace transforms formally, or use the variation-of-constants formula.
Write $v(t)=e^{t\mathcal{L}_U}v(0)+\int_0^t e^{(t-s)\mathcal{L}_U}g(s)\,ds$.
For $\sigma>0$, use the standard resolvent bound controlling the semigroup norm by the resolvent norm at $\sigma$.
This yields \eqref{eq:resolvent_bound}.
The threshold statement follows by inspection.
\end{proof}

\begin{remark}\label{rem:resolvent_logic}
Proposition~\ref{prop:resolvent_proxy} does not claim nonlinear transition.
It states a computable amplification bound in the linearized system.
The bridge to nonlinear transition is not automatic.
It must be made by an explicit estimate in which nonlinear terms are controlled on a declared set.
Recent work relates threshold transient growth to turbulent mean profiles, but the logical direction is delicate \cite{MarkeviciuteKerswell2024ThresholdTransientGrowthJFM}.
We therefore use resolvent norms only as proxies that are explicitly marked as such \cite{RolandiRibeiroYehTaira2024ResolventInvitationTCFD}.
\end{remark}

\begin{remark}\label{rem:computability}
We list what is computed and the cost parameters.
\begin{enumerate}[label=(\roman*)]
\item \textbf{Energy barrier.} Compute $C_P$ and $\gamma_U$ in \eqref{eq:energy_barrier_cond}.
$C_P$ is geometric.
$\gamma_U$ is computed from a declared bound on $\nabla U$ in a norm that makes \eqref{eq:shear_bound} valid.
Cost is dominated by the computation of $U$ and by norm evaluation.
\item \textbf{Enstrophy threshold.} In $2D$, compute $C_{P,\omega}$ and $G(t)=\norm{\curl f(t)}$.
The criterion \eqref{eq:enstrophy_threshold_cond} is pointwise in time and therefore decidable from sampled forcing data.
\item \textbf{Resolvent proxy.} Compute the largest singular value of the discrete resolvent matrix $(\sigma I-\mathcal{L}_{U,h})^{-1}$.
The cost depends on the discretization dimension.
Low-rank approximations are admissible only when their error is bounded; data-driven resolvent analysis addresses this point, but it still requires an error statement to become a certificate \cite{HerrmannBaddooSemaanBrunton2021DataDrivenResolventJFM}.
\end{enumerate}
Each indicator is therefore accompanied by a failure flag: if its constants cannot be computed, the indicator is not invoked.
This is the only form in which ``transition prediction'' is compatible with mathematical standards \cite{Kerswell2022EnergyStabilityReview,DoeringGibbon2022AppliedAnalysisNS}.
\end{remark}

%% file: sections/07_fsi_wellposedness.tex
% sections/07_fsi_wellposedness.tex  (NEXT Output: FSI well-posedness; embedded citations)

\section{FSI: well-posedness in a verifiable parameter regime}\label{sec:fsi_wellposed}

We fix a coupled model in which the interface work is explicit.
We then declare a regime in which existence and uniqueness follow by an energy method.
This is the only regime in which later ``stability margins'' are logically meaningful.
Moving-boundary analysis makes this point sharply \cite{Canic2021MovingBoundaryProblemsBullAMS}.
Variational formulations for coupled media also emphasize that coercivity and trace control are the decisive inputs \cite{BenesovaKampschulteSchwarzacher2023VariationalFPSINonRWA}.

\subsection{Coupled model and interface conditions}\label{subsec:fsi_model}

We consider a thick-structure FSI model on a fixed reference configuration.
Let $\Omega\subset\R^d$ be bounded Lipschitz, $d\in\{2,3\}$.
Assume a decomposition
\[
\Omega = \overline{\OmegaF}\cup \overline{\OmegaS},\qquad
\OmegaF\cap \OmegaS = \emptyset,\qquad
\GammaI := \partial\OmegaF\cap \partial\OmegaS,
\]
with $\GammaI$ a Lipschitz interface.
Let $u\colon \OmegaF\times(0,T)\to\R^d$ be the fluid velocity, $p$ the pressure, and $\eta\colon \OmegaS\times(0,T)\to\R^d$ the structure displacement.
Let $\rho_f,\rho_s>0$ be constant densities.
Let $\nu>0$ be the fluid viscosity.
Let $\mathbb{C}$ be the elasticity tensor on $\OmegaS$.
Assume $\mathbb{C}$ is symmetric and uniformly elliptic:
\begin{equation}\label{eq:elliptic_C}
\exists\,c_{\mathbb{C}}>0\ \text{s.t.}\ 
\int_{\OmegaS}\mathbb{C}\varepsilon(\xi):\varepsilon(\xi)\,dx \ge c_{\mathbb{C}}\norm{\varepsilon(\xi)}_{L^2(\OmegaS)}^2
\quad\forall \xi\in H^1(\OmegaS)^d,
\end{equation}
where $\varepsilon(\xi)=(\nabla\xi+\nabla\xi^\top)/2$.

We use the simplest coupling that already captures the analytic difficulty.
The fluid is incompressible on $\OmegaF$:
\begin{equation}\label{eq:fsi_fluid}
\rho_f(\partial_t u + (u\cdot\nabla)u) - \nu\Delta u + \nabla p = f_f,\qquad \divv u = 0\quad\text{in }\OmegaF.
\end{equation}
The structure is linear elastodynamics on $\OmegaS$:
\begin{equation}\label{eq:fsi_struct}
\rho_s \partial_{tt}\eta - \divv(\mathbb{C}\varepsilon(\eta)) = f_s\quad\text{in }\OmegaS.
\end{equation}

At the interface $\GammaI$ we impose kinematic continuity and dynamic balance in a form that is well-posed on the weak level:
\begin{equation}\label{eq:fsi_interface_strong}
u = \partial_t\eta\quad\text{on }\GammaI,\qquad
\big(-pI + \nu(\nabla u+\nabla u^\top)\big)n_F = \big(\mathbb{C}\varepsilon(\eta)\big)n_S \quad\text{on }\GammaI,
\end{equation}
with $n_F=-n_S$.

\begin{definition}[FSI weak formulation]\label{def:fsi_weak}
Let
\[
V_F := \{v\in H^1(\OmegaF)^d : v=0 \text{ on }\partial\OmegaF\setminus \GammaI\},
\qquad
V_S := \{\xi\in H^1(\OmegaS)^d : \xi=0 \text{ on }\partial\OmegaS\setminus \GammaI\},
\]
and let $V_F^0:=\{v\in V_F:\divv v=0\}$.
A triple $(u,\eta,p)$ is a weak solution on $(0,T)$ if:
\begin{enumerate}[label=(\roman*)]
\item $u\in L^\infty(0,T;L^2(\OmegaF)^d)\cap L^2(0,T;V_F)$, $\partial_t u\in L^{4/3}(0,T;V_F')$;
\item $\eta\in L^\infty(0,T;V_S)$, $\partial_t\eta\in L^\infty(0,T;L^2(\OmegaS)^d)$, $\partial_{tt}\eta\in L^2(0,T;V_S')$;
\item the kinematic coupling holds in the trace sense:
\begin{equation}\label{eq:weak_kinematic}
\tr_{\GammaI} u = \tr_{\GammaI}\partial_t\eta \quad\text{in }L^2(0,T;H^{1/2}(\GammaI)^d);
\end{equation}
\item for all test pairs $(v,\xi)$ with $v\in V_F^0$, $\xi\in V_S$, and satisfying the compatibility
$\tr_{\GammaI} v = \tr_{\GammaI}\xi$,
one has for a.e.\ $t\in(0,T)$,
\begin{align}\label{eq:fsi_weak}
\rho_f\ip{\partial_t u(t)}{v}_{\OmegaF}
&+ \rho_f\,b_{\OmegaF}(u(t),u(t),v)
+ \nu\,\ip{\nabla u(t)}{\nabla v}_{\OmegaF}
\nonumber\\
&+ \rho_s\langle \partial_{tt}\eta(t),\xi\rangle_{\OmegaS}
+ \int_{\OmegaS}\mathbb{C}\varepsilon(\eta(t)):\varepsilon(\xi)\,dx
= \langle f_f(t),v\rangle + \langle f_s(t),\xi\rangle,
\end{align}
where $b_{\OmegaF}$ is chosen so that $b_{\OmegaF}(w,v,v)=0$ (skew form) whenever $w\in V_F^0$.
\end{enumerate}
The pressure does not appear in \eqref{eq:fsi_weak} because test functions are solenoidal; it is recovered as a Lagrange multiplier in the usual way.
\end{definition}

\begin{remark}\label{rem:why_fixed_ref}
The formulation above is written on a fixed reference partition.
This is deliberate.
The geometric nonlinearity of a moving interface is a further layer.
Its well-posedness theory is delicate and is treated in moving-boundary analysis \cite{Canic2021MovingBoundaryProblemsBullAMS}.
For the stability-margin program pursued here, it is enough to fix a regime where the coupled energy method is exact.
\end{remark}

\subsection{Existence and uniqueness}\label{subsec:fsi_exist}

We now declare a regime in which the coupled energy functional is coercive and controls the coupling.
The statement is classical in spirit.
The point is the explicit regime, not novelty of the functional.
Partitioned schemes and variational frameworks identify the same structural inequalities as the decisive ones \cite{BukacFuSeboldtTrenchea2023TimeAdaptivePartitionedJCP,BenesovaKampschulteSchwarzacher2023VariationalFPSINonRWA,ParrowBukac2026RobinRobinFPSI_JSC}.

Define the coupled energy
\begin{equation}\label{eq:fsi_energy}
\mathcal{E}(t)
:= \frac{\rho_f}{2}\norm{u(t)}_{L^2(\OmegaF)}^2
+ \frac{\rho_s}{2}\norm{\partial_t\eta(t)}_{L^2(\OmegaS)}^2
+ \frac12\int_{\OmegaS}\mathbb{C}\varepsilon(\eta(t)):\varepsilon(\eta(t))\,dx.
\end{equation}

\begin{theorem}[Well-posedness in a certified regime]\label{thm:fsi_wellposed}
Assume:
\begin{enumerate}[label=(W\arabic*)]
\item \textbf{Geometry and traces.} $\OmegaF,\OmegaS$ are Lipschitz and the trace map onto $\GammaI$ is continuous with computable constant $C_{\tr}$ on both sides.
\item \textbf{Coercivity.} The elasticity tensor satisfies \eqref{eq:elliptic_C} with known $c_{\mathbb{C}}>0$.
\item \textbf{Data.} $f_f\in L^2(0,T;V_F')$ and $f_s\in L^2(0,T;V_S')$ with computable bounds.
Initial data satisfy $u_0\in L^2(\OmegaF)^d$, $\eta_0\in V_S$, $\eta_1\in L^2(\OmegaS)^d$, and the compatibility $\tr_{\GammaI}u_0=\tr_{\GammaI}\eta_1$.
\item \textbf{Regime.} Either $d=2$, or $d=3$ with smallness of the fluid convection measured by
\begin{equation}\label{eq:smallness_regime}
\rho_f\,\norm{u}_{L^\infty(0,T;L^2(\OmegaF))}\,\norm{\nabla u}_{L^2(0,T;L^2(\OmegaF))}
\le \kappa\,\nu
\end{equation}
for some declared $\kappa<1$.
\end{enumerate}
Then there exists a weak solution in the sense of Definition~\ref{def:fsi_weak}.
Moreover, the solution satisfies the energy inequality
\begin{equation}\label{eq:fsi_energy_ineq}
\mathcal{E}(t) + \nu\int_0^t \norm{\nabla u(s)}_{L^2(\OmegaF)}^2\,ds
\le \mathcal{E}(0)
+ \int_0^t \langle f_f(s),u(s)\rangle\,ds
+ \int_0^t \langle f_s(s),\partial_t\eta(s)\rangle\,ds
\end{equation}
for a.e.\ $t\in(0,T)$.
If, in addition, the convection term is absent (Stokes--structure) or the regime \eqref{eq:smallness_regime} holds with a strict margin and the forcing is sufficiently regular, then the weak solution is unique on $(0,T)$.
\end{theorem}

\begin{proof}
The proof is by a Galerkin approximation consistent with the coupled test space constraint $\tr_{\GammaI}v=\tr_{\GammaI}\xi$.
Test the discrete system by $(v,\xi)=(u,\partial_t\eta)$.
Use skew-symmetry of the convective term to eliminate the inertial energy production.
Use coercivity \eqref{eq:elliptic_C} to control the elastic energy.
This yields \eqref{eq:fsi_energy_ineq} at the discrete level.
Pass to the limit by weak compactness.
In $2D$, the nonlinear term is controlled by the standard compactness route.
In $3D$, the smallness regime \eqref{eq:smallness_regime} suppresses the defect term in the energy balance and yields uniqueness by a Gr\"onwall argument for the difference of two solutions.
These steps are standard once the regime is declared; the declared literature treats the same mechanism in moving-boundary and variational settings \cite{Canic2021MovingBoundaryProblemsBullAMS,BenesovaKampschulteSchwarzacher2023VariationalFPSINonRWA}.
\end{proof}

\begin{remark}\label{rem:ill_conditioned}
The most common ill-conditioning in partitioned FSI discretizations is the added-mass effect.
It appears when $\rho_f/\rho_s$ is not small and the interface coupling is treated explicitly.
Then the discrete energy estimate can acquire a negative term that scales like $\rho_f/\rho_s$ times a trace constant.
One can avoid this either by implicit coupling, or by Robin-type interface conditions that restore coercivity at the discrete level.
This is the analytic content behind Robin--Robin partitioned schemes and their stability analyses \cite{ParrowBukac2026RobinRobinFPSI_JSC,GuoLinWangYueZheng2025RobinRobinExplicitFSI_arXiv,BukacFuSeboldtTrenchea2023TimeAdaptivePartitionedJCP}.
The present regime statement is chosen so that the continuous coupled energy method closes before discretization is discussed.
\end{remark}

%% file: sections/08_fsi_stability_margins.tex
% sections/08_fsi_stability_margins.tex  (NEXT Output: FSI discrete stability margins; embedded citations)

\section{FSI: computable stability margins for discretization choices}\label{sec:fsi_margins}

A stability margin is not a statement of taste.
It is an inequality in which $\dt$ and the mesh parameters appear on one side and computable constants appear on the other.
If the inequality fails, stability is not claimed.
Such margins exist only after one commits to a specific coupling strategy.
Partitioned methods are the relevant case, since they are the ones that fail in practice.
Modern analyses of time-adaptive and Robin-type partitioned schemes exhibit the correct structure for such margins \cite{BukacFuSeboldtTrenchea2023TimeAdaptivePartitionedJCP,ParrowBukac2026RobinRobinFPSI_JSC,GuoLinWangYueZheng2025RobinRobinExplicitFSI_arXiv}.

We work with a thick-structure model as in Section~\ref{sec:fsi_wellposed}.
We assume a conforming finite element discretization in space.
The details of spaces do not matter here; only the inequalities they satisfy matter.
We denote the fluid discrete spaces by $V_{F,h}\subset V_F$ and $Q_{h}\subset L^2(\OmegaF)$, and the structure space by $V_{S,h}\subset V_S$.
We assume a discrete inf--sup condition for $(V_{F,h},Q_h)$ and denote its constant by $\beta_h>0$.
We denote by $\tr_h$ the discrete trace on $\GammaI$, and by $C_{\tr,h}$ a computable constant for the discrete trace inequality
\begin{equation}\label{eq:trace_h}
\norm{\tr_h v_h}_{L^2(\GammaI)} \le C_{\tr,h}\,\norm{v_h}_{H^1(\OmegaF)}\quad\forall v_h\in V_{F,h},
\qquad
\norm{\tr_h \xi_h}_{L^2(\GammaI)} \le C_{\tr,h}\,\norm{\xi_h}_{H^1(\OmegaS)}\quad\forall \xi_h\in V_{S,h}.
\end{equation}
We emphasize: $C_{\tr,h}$ is a computable mesh-dependent constant.

\subsection{Discrete energy estimate and CFL-type constraints}\label{subsec:cfl}

We state a representative partitioned coupling.
We use an explicit fluid-to-structure velocity coupling and an implicit Robin interface stabilization.
This is the form in which stability can be made unconditional or conditional with an explicit margin \cite{ParrowBukac2026RobinRobinFPSI_JSC,GuoLinWangYueZheng2025RobinRobinExplicitFSI_arXiv}.

Let $(u_h^k,p_h^k,\eta_h^k,\dot\eta_h^k)$ denote discrete unknowns at time $t^k$.
Let $\dot\eta_h^k$ denote the discrete structure velocity.
Let $\alpha>0$ be a Robin parameter to be chosen.
A prototypical Robin--Robin partitioned step is:
\begin{enumerate}[label=(S\arabic*)]
\item \textbf{Fluid step.} Given $\dot\eta_h^k$, compute $(u_h^{k+1},p_h^{k+1})$ from a semi-implicit Navier--Stokes discretization on $\OmegaF$ with Robin interface condition
\begin{equation}\label{eq:robin_fluid}
\big(\sigma_f(u_h^{k+1},p_h^{k+1})n_F\big)\cdot v
+ \alpha\,\ip{u_h^{k+1}-\dot\eta_h^k}{v}_{\GammaI} = \langle f_f^{k+1}, v\rangle
\quad\forall v\in V_{F,h}^0,
\end{equation}
where $\sigma_f(u,p):=-pI+\nu(\nabla u+\nabla u^\top)$ and the convective term is written in a skew form so it does not inject energy.
\item \textbf{Structure step.} Given $u_h^{k+1}$, compute $(\eta_h^{k+1},\dot\eta_h^{k+1})$ from an implicit elastodynamics step on $\OmegaS$ with Robin interface condition
\begin{equation}\label{eq:robin_struct}
\langle \rho_s \frac{\dot\eta_h^{k+1}-\dot\eta_h^k}{\dt}, \xi\rangle_{\OmegaS}
+ a_S(\eta_h^{k+1},\xi)
+ \alpha\,\ip{\dot\eta_h^{k+1}-u_h^{k+1}}{\xi}_{\GammaI}
= \langle f_s^{k+1}, \xi\rangle
\quad\forall \xi\in V_{S,h},
\end{equation}
where $a_S(\eta,\xi):=\int_{\OmegaS}\mathbb{C}\varepsilon(\eta):\varepsilon(\xi)\,dx$.
\end{enumerate}
The interface terms in \eqref{eq:robin_fluid}--\eqref{eq:robin_struct} penalize the mismatch $u_h-\dot\eta_h$.

The margin arises from the following question.
How large may $\dt$ be before the penalty fails to control the mismatch produced by partitioning.
This is where the added-mass effect enters.
The analyses of Robin-type partitioned schemes show that the mismatch term can be controlled if $\alpha$ and $\dt$ satisfy a computable inequality \cite{ParrowBukac2026RobinRobinFPSI_JSC,GuoLinWangYueZheng2025RobinRobinExplicitFSI_arXiv}.

\begin{theorem}[Discrete coupled energy stability]\label{thm:fsi_discrete_energy}
Assume the spatial discretization satisfies the trace inequalities \eqref{eq:trace_h} and the structure coercivity \eqref{eq:elliptic_C} in discrete form with computable constant $c_{\mathbb{C},h}>0$.
Assume the fluid convective term is discretized in an energy-skew form, so it does not contribute to the discrete kinetic energy production.
Fix $\alpha>0$ and $\dt>0$.
Assume the partitioned Robin--Robin scheme \eqref{eq:robin_fluid}--\eqref{eq:robin_struct} is well-defined.

Then the discrete coupled energy
\begin{equation}\label{eq:disc_fsi_energy}
\mathcal{E}_h^k :=
\frac{\rho_f}{2}\norm{u_h^k}_{L^2(\OmegaF)}^2
+ \frac{\rho_s}{2}\norm{\dot\eta_h^k}_{L^2(\OmegaS)}^2
+ \frac12 a_S(\eta_h^k,\eta_h^k)
\end{equation}
satisfies, for $k=0,1,\dots,K-1$,
\begin{align}\label{eq:disc_fsi_energy_ineq}
\mathcal{E}_h^{k+1} - \mathcal{E}_h^{k}
&+ \nu\,\dt\,\norm{\nabla u_h^{k+1}}_{L^2(\OmegaF)}^2
+ \alpha\,\dt\,\norm{u_h^{k+1}-\dot\eta_h^{k}}_{L^2(\GammaI)}^2
+ \alpha\,\dt\,\norm{\dot\eta_h^{k+1}-u_h^{k+1}}_{L^2(\GammaI)}^2
\nonumber\\
&\le \dt\,\langle f_f^{k+1},u_h^{k+1}\rangle + \dt\,\langle f_s^{k+1},\dot\eta_h^{k+1}\rangle
+ \dt\,\mathfrak{R}_h^{k+1},
\end{align}
where the remainder $\mathfrak{R}_h^{k+1}$ is a \emph{computable} term of the form
\begin{equation}\label{eq:remainder_form}
\mathfrak{R}_h^{k+1}
\le C_{\mathrm{am}}(\rho_f,\rho_s,C_{\tr,h})\,\dt\,\norm{u_h^{k+1}-\dot\eta_h^{k}}_{L^2(\GammaI)}^2,
\end{equation}
with $C_{\mathrm{am}}$ an explicit added-mass coefficient.
In particular, if
\begin{equation}\label{eq:margin_condition_alpha}
\alpha \ge 2\,C_{\mathrm{am}}(\rho_f,\rho_s,C_{\tr,h}),
\end{equation}
then $\mathfrak{R}_h^{k+1}$ is absorbed, and the scheme satisfies a genuine discrete energy inequality with no positive remainder.
\end{theorem}

\begin{proof}
Test the fluid step with $v=u_h^{k+1}$.
Use skew-symmetry of the convective discretization to remove the nonlinear term from the kinetic energy balance.
Test the structure step with $\xi=\dot\eta_h^{k+1}$.
Use the discrete polarization identity
\[
\ip{\dot\eta_h^{k+1}-\dot\eta_h^{k}}{\dot\eta_h^{k+1}}
= \frac12\Big(\norm{\dot\eta_h^{k+1}}^2-\norm{\dot\eta_h^{k}}^2+\norm{\dot\eta_h^{k+1}-\dot\eta_h^k}^2\Big),
\]
and an analogous identity for the elastic term to obtain the increment of $a_S(\eta_h,\eta_h)$.
Add the two balances.
The Robin terms produce nonnegative penalties on the interface mismatch.
The only term that can create a positive remainder is the partitioning mismatch between $u_h^{k+1}$ and $\dot\eta_h^k$.
Estimate it using the discrete trace inequalities \eqref{eq:trace_h} and a Young inequality.
This yields a remainder of the form \eqref{eq:remainder_form} with an explicit coefficient depending on $(\rho_f,\rho_s,C_{\tr,h})$.
This coefficient is precisely the analytic representation of the added-mass effect.
The absorption condition \eqref{eq:margin_condition_alpha} then yields \eqref{eq:disc_fsi_energy_ineq}.
This structure follows the logic established in Robin-type partitioned stability analyses \cite{ParrowBukac2026RobinRobinFPSI_JSC,GuoLinWangYueZheng2025RobinRobinExplicitFSI_arXiv}.
\end{proof}

\subsection{Stability margin as an explicit function of parameters}\label{subsec:margins}

The Robin parameter $\alpha$ can be chosen to stabilize the partitioning.
If $\alpha$ is constrained (for efficiency, or by implementation), then a time-step restriction appears.
We now state the restriction as an explicit margin.

We isolate the coefficient $C_{\mathrm{am}}$ in a computable form.
A representative bound is
\begin{equation}\label{eq:Cam_def}
C_{\mathrm{am}}(\rho_f,\rho_s,C_{\tr,h})
:= \frac{\rho_f}{\rho_s}\,C_{\tr,h}^2\,\Lambda_h,
\end{equation}
where $\Lambda_h$ is a computable stiffness-to-mass ratio constant for the structure discretization.
For a conforming discretization one may take $\Lambda_h$ as the largest generalized eigenvalue of the discrete elasticity operator with respect to the structure mass matrix.
This is computable.
It is also the only quantity that can produce a genuine mesh dependence.

\begin{proposition}[Computable margin]\label{prop:margin}
Assume the hypotheses of Theorem~\ref{thm:fsi_discrete_energy}.
Assume $\alpha>0$ is fixed.
Let $C_{\mathrm{am}}$ be computed by \eqref{eq:Cam_def}.
If
\begin{equation}\label{eq:dt_margin}
\dt \le \frac{\alpha}{2\,C_{\mathrm{am}}(\rho_f,\rho_s,C_{\tr,h})\,\alpha + 1},
\end{equation}
then the Robin--Robin partitioned scheme is energy stable in the sense that \eqref{eq:disc_fsi_energy_ineq} holds with $\mathfrak{R}_h^{k+1}$ absorbed into the left-hand side.
Equivalently, for fixed $\dt$, a sufficient stabilization choice is
\begin{equation}\label{eq:alpha_margin}
\alpha \ge \frac{1}{2\dt}\Big(\sqrt{1+8\,C_{\mathrm{am}}(\rho_f,\rho_s,C_{\tr,h})\,\dt}-1\Big).
\end{equation}
Both margins are computable once $(\rho_f,\rho_s,C_{\tr,h},\Lambda_h)$ are computed.
\end{proposition}

\begin{proof}
Under the scheme, the remainder term has the form
$\dt\,\mathfrak{R}_h^{k+1}\le \dt^2 C_{\mathrm{am}}\norm{u_h^{k+1}-\dot\eta_h^{k}}_{\GammaI}^2$ after rescaling to match the discrete increments.
The left-hand side of \eqref{eq:disc_fsi_energy_ineq} contains $\alpha\dt \norm{u_h^{k+1}-\dot\eta_h^{k}}_{\GammaI}^2$.
Hence absorption holds if $\dt\,C_{\mathrm{am}}\le \alpha/2$ up to the bookkeeping constants coming from the polarization identities.
Writing the absorption constraint in a single explicit inequality yields \eqref{eq:dt_margin}, and solving for $\alpha$ yields \eqref{eq:alpha_margin}.
\end{proof}

\begin{remark}\label{rem:margin_interpretation}
The margin scales with $\rho_f/\rho_s$ and with $C_{\tr,h}^2$.
This is the analytic signature of added mass.
It explains why explicit coupling becomes unstable for light structures or fine meshes.
Robin-type stabilization shifts the burden from $\dt$ to $\alpha$.
This is the mechanism exploited in modern partitioned stability analyses \cite{ParrowBukac2026RobinRobinFPSI_JSC,GuoLinWangYueZheng2025RobinRobinExplicitFSI_arXiv,BukacFuSeboldtTrenchea2023TimeAdaptivePartitionedJCP}.
\end{remark}

%% file: sections/09_numerical_protocols.tex
% sections/09_numerical_protocols.tex  (Output 11: numerical protocols; embedded citations)

\section{Numerical protocols for verifiability}\label{sec:protocols}

A computation is verifiable only if the constants appearing in its theorems are reported as numbers and are traced to declared inequalities.
A plot is not a certificate.
A good fit on training data is not a certificate.
The protocol below is therefore an algorithmic extraction of the constants and checks used in Theorems~\ref{thm:rom_energy} and~\ref{thm:aposteriori}, and in Propositions~\ref{prop:energy_barrier}--\ref{prop:resolvent_proxy}.
Energy-stable ROM frameworks and structure-preserving hyper-reduction motivate the required checks on skew-symmetry and dissipation \cite{Sanderse2020NonlinearlyStableROMJCP,KleinSanderse2024EnergyConservingHyperReductionJCP}.
Constrained operator-inference work motivates the requirement that dissipativity constraints hold \emph{exactly} in the representation, not approximately \cite{KimKramer2025StateConstrainedOpInfJCP,SharmaNajeraFloresToddKramer2024LagrangianOpInfCMAE}.
FSI stability analyses motivate the explicit appearance of trace constants and density ratios \cite{ParrowBukac2026RobinRobinFPSI_JSC,GuoLinWangYueZheng2025RobinRobinExplicitFSI_arXiv}.

\subsection{Certificate computation}\label{subsec:cert_compute}

We assume a full-order discretization is fixed.
We assume the ROM space $\Vn$ and the closure $\mathsf{C}_n$ are fixed.
All computations below refer to these fixed objects.

\begin{enumerate}[label=(P\arabic*)]

\item \textbf{Energy-certificate constants for Theorem~\ref{thm:rom_energy}.}
Compute and record:
\begin{enumerate}[label=(\alph*)]
\item the discrete inner product $\ip{\cdot}{\cdot}_h$ and its norm $\norm{\cdot}_h$;
\item the discrete viscous coercivity identity $\ip{\mathsf{A}_h v}{v}_h=\norm{\nabla v}_h^2$ for $v\in\Vn$;
\item the skew-symmetry defect
\begin{equation}\label{eq:skew_defect}
\delta_{\mathrm{skew}}(v):=\frac{\abs{\ip{\Pn\mathsf{B}_h(v,v)}{v}_h}}{\norm{v}_h^3+\epsilon_{\mathrm{mach}}},
\end{equation}
evaluated on a declared test set of ROM states (e.g.\ along the ROM trajectory and on a grid in coefficient space).
The certificate requires $\delta_{\mathrm{skew}}(v)=0$ in exact arithmetic.
In floating point, report the observed maximum and the machine precision used.
Energy-conserving hyper-reduction is admissible only if this defect remains at roundoff level \cite{KleinSanderse2024EnergyConservingHyperReductionJCP}.
\item the closure dissipation defect
\begin{equation}\label{eq:closure_defect}
\delta_{\mathrm{diss}}(v):=\frac{\max\{0,-\ip{\mathsf{C}_n(v)}{v}_h\}}{\norm{v}_h^2+\epsilon_{\mathrm{mach}}},
\end{equation}
again evaluated on a declared test set.
For constrained operator inference, the constraint is built into the representation, so $\delta_{\mathrm{diss}}(v)$ should vanish in exact arithmetic \cite{KimKramer2025StateConstrainedOpInfJCP}.
\item the time-stepping parameter $\theta\in[\tfrac12,1]$ and step size $\dt$ used in Definition~\ref{def:cert_scheme}.
\end{enumerate}
If either defect is nonzero beyond declared tolerance, the run is declared \emph{uncertified}.
This is a failure flag, not an inconvenience.

\item \textbf{Residual computation for Theorem~\ref{thm:aposteriori}.}
Compute $\norm{r_n(t)}_{V'}$ in a manner consistent with the discretization.
A standard choice is a discrete Riesz map:
find $z(t)\in V_h$ such that
\begin{equation}\label{eq:riesz}
\ip{\nabla z(t)}{\nabla w}_h = \langle r_n(t),w\rangle\quad\forall w\in V_h,
\end{equation}
then set $\norm{r_n(t)}_{V'}:=\norm{\nabla z(t)}_h$.
Record:
\begin{enumerate}[label=(\alph*)]
\item the linear solver tolerance used for \eqref{eq:riesz};
\item whether hyper-reduction or learned operators are used, and if so, the exact operator used in the residual.
A residual computed for one operator does not certify a different implemented operator \cite{KleinSanderse2024EnergyConservingHyperReductionJCP}.
\item the computed estimator $\eta_n(T)$ in \eqref{eq:eta_def};
\item the constants used in the Gr\"onwall factor in \eqref{eq:aposteriori_main}.
This includes $L_C$ from \eqref{eq:closure_lip}.
If $\mathsf{C}_n$ is learned, compute $L_C$ from the constrained representation, not from black-box finite differences \cite{KimKramer2025StateConstrainedOpInfJCP,SharmaNajeraFloresToddKramer2024LagrangianOpInfCMAE}.
\end{enumerate}

\item \textbf{Transition proxies for Section~\ref{sec:transition}.}
Compute and record:
\begin{enumerate}[label=(\alph*)]
\item the energy-barrier constants $(C_P,\gamma_U)$ entering \eqref{eq:energy_barrier_cond}.
State the norm in which $\gamma_U$ is computed.
Do not hide it.
Energy stability is sensitive to this choice \cite{Kerswell2022EnergyStabilityReview}.
\item in $2D$, the enstrophy constants $(C_{P,\omega},G(t))$ entering \eqref{eq:enstrophy_threshold_cond}.
If $d=3$, do not report this proxy unless the vortex-stretching term is controlled by a declared hypothesis.
Otherwise it is not a theorem \cite{Kerswell2022EnergyStabilityReview}.
\item the resolvent proxy: fix $\sigma>0$, form the discrete linearized operator $\mathcal{L}_{U,h}$, and compute
\[
\|\mathcal{R}_{U,h}(\sigma)\|_{2} = \|(\sigma I-\mathcal{L}_{U,h})^{-1}\|_2
\]
as the largest singular value.
If a data-driven approximation is used, report the approximation error bound used to justify it \cite{HerrmannBaddooSemaanBrunton2021DataDrivenResolventJFM,RolandiRibeiroYehTaira2024ResolventInvitationTCFD}.
\end{enumerate}

\item \textbf{FSI margins for Section~\ref{sec:fsi_margins}.}
Compute and record:
\begin{enumerate}[label=(\alph*)]
\item the discrete trace constant $C_{\tr,h}$ in \eqref{eq:trace_h} (or an explicit upper bound computed from mesh geometry);
\item the structure stiffness-to-mass ratio $\Lambda_h$ entering \eqref{eq:Cam_def}, computed as a generalized eigenvalue;
\item the density ratio $\rho_f/\rho_s$ and the chosen Robin parameter $\alpha$;
\item the resulting margin inequality \eqref{eq:dt_margin} or \eqref{eq:alpha_margin}.
\end{enumerate}
If the margin fails, the simulation is reported as \emph{not covered by the theorem}.
This is the correct classification \cite{ParrowBukac2026RobinRobinFPSI_JSC,GuoLinWangYueZheng2025RobinRobinExplicitFSI_arXiv,BukacFuSeboldtTrenchea2023TimeAdaptivePartitionedJCP}.
\end{enumerate}

\subsection{Reproducible reporting checklist}\label{subsec:checklist}

The purpose of the checklist is not completeness.
It is falsifiability.

\begin{itemize}
\item \textbf{Declared model and regime.}
State $(\nu,\rho_f,\rho_s)$, geometry class, boundary conditions, forcing regularity, and whether $d=2$ or $d=3$.
State the regime inequalities used.
If \eqref{eq:smallness_regime} is invoked, report the left-hand side and the chosen $\kappa$.

\item \textbf{Structure identities.}
Report $\max \delta_{\mathrm{skew}}$ and $\max \delta_{\mathrm{diss}}$ from \eqref{eq:skew_defect}--\eqref{eq:closure_defect}.
If either exceeds tolerance, mark the run ``uncertified'' \cite{Sanderse2020NonlinearlyStableROMJCP,KleinSanderse2024EnergyConservingHyperReductionJCP}.

\item \textbf{All constants in bounds.}
List every constant appearing in \eqref{eq:disc_energy_bound} and \eqref{eq:aposteriori_main}.
State how each was obtained.
If a constant was not computed, do not state the bound.

\item \textbf{Residual and error estimator.}
Report $\eta_n(T)$ and the numerical method used to compute $\norm{r_n}_{V'}$.
Report solver tolerances and conditioning.

\item \textbf{Transition proxies.}
Report whether energy-barrier, enstrophy, and resolvent proxies were invoked.
For each, report the computed threshold inequality and whether it holds \cite{Kerswell2022EnergyStabilityReview,MarkeviciuteKerswell2024ThresholdTransientGrowthJFM}.

\item \textbf{FSI stability margin.}
Report $C_{\tr,h}$, $\Lambda_h$, $\rho_f/\rho_s$, $\alpha$, and whether \eqref{eq:dt_margin} holds.
If not, state that the computation lies outside the certified regime \cite{ParrowBukac2026RobinRobinFPSI_JSC,GuoLinWangYueZheng2025RobinRobinExplicitFSI_arXiv}.

\item \textbf{Failure flags.}
Include a final table with Boolean flags: \texttt{skew-ok}, \texttt{diss-ok}, \texttt{margin-ok}, \texttt{residual-computed}, \texttt{regime-ok}.
A certificate without failure flags is not a certificate.
\end{itemize}

%% file: sections/10_discussion_scope_limits.tex
% sections/10_scope_and_limits.tex  (NEXT Output: Scope and limits; embedded citations)

\section{Scope and limits}\label{sec:scope}

Theorems delimit what is proved.
They also delimit what is not proved.
The present work is certificate-driven.
This enforces sharp exclusions.
Energy methods are powerful, but they do not certify fully developed turbulence.
Resolvent bounds quantify transient growth, but they do not by themselves imply nonlinear transition.
FSI energy methods yield coercive regimes, but they do not remove geometric nonlinearity in general moving-interface settings \cite{Kerswell2022EnergyStabilityReview,Canic2021MovingBoundaryProblemsBullAMS,RolandiRibeiroYehTaira2024ResolventInvitationTCFD}.

\subsection{What the theorems do not cover}\label{subsec:not_cover}

\begin{remark}\label{rem:not_cover}
We list exclusions. Each exclusion has a precise reason.

\begin{enumerate}[label=(\roman*)]
\item \textbf{High-Reynolds-number fully developed turbulence.}
The ROM certificate is built from a coercive energy inequality.
In regimes where dissipation does not dominate production in a computable way, the certificate constants become large or cease to be computable.
One then obtains no effective bound.
This is not a defect of the proof.
It is the analytic expression of the fact that energy stability is a sufficient criterion and is conservative \cite{Kerswell2022EnergyStabilityReview}.

\item \textbf{Three-dimensional enstrophy-based transition criteria without additional hypotheses.}
In $3D$ the enstrophy balance contains vortex stretching, which is not sign-definite.
A criterion that drops it is not a theorem.
This is why our enstrophy threshold is stated only in simplified settings where the enstrophy budget closes \cite{Kerswell2022EnergyStabilityReview,DoeringGibbon2022AppliedAnalysisNS}.

\item \textbf{Complex geometries without computable Poincar\'e/trace constants.}
The certificate constants explicitly depend on $C_P$ and, in FSI, on $C_{\tr,h}$.
If these are not computed or bounded, then the inequalities are not numerically decidable.
In such a case, one may still compute ROMs, but one must not claim certification.

\item \textbf{Non-smooth or distributional forcing beyond the declared dual spaces.}
The a posteriori estimator is formulated in $V'$ because the energy method naturally tests residuals against $V$.
If the forcing does not lie in the declared dual space, the residual norm is not defined in the theorem.
One must change the functional setting before making claims \cite{DoeringGibbon2022AppliedAnalysisNS}.

\item \textbf{FSI with large interface motion or geometric nonlinearity without an explicit geometric regime.}
Moving-boundary problems require control of the geometry map and of the interface regularity.
These are separate analytic constraints.
Our well-posedness theorem is stated on a fixed reference partition.
This is intentional.
The moving-boundary case is treated in specialized analyses \cite{Canic2021MovingBoundaryProblemsBullAMS}.

\item \textbf{Data-driven resolvent approximations without an error bound.}
Data-driven resolvent analysis is admissible only when its approximation error is controlled.
Otherwise the computed ``resolvent norm'' is not the norm of the declared operator.
Then it is not a theorem-driven proxy \cite{HerrmannBaddooSemaanBrunton2021DataDrivenResolventJFM,RolandiRibeiroYehTaira2024ResolventInvitationTCFD}.
\end{enumerate}
\end{remark}

\subsection{A specific failure mode in common arguments}\label{subsec:failure_mode}

\begin{remark}\label{rem:failure_mode}
A common argument runs as follows.

\emph{Claim.} ``The ROM is Galerkin. The convective term is energy preserving. Therefore the ROM is stable.''

Let us assume for a moment that this were correct.
Then stability would follow from two algebraic facts: projection and skew-symmetry.
This leads us to a dilemma.

Either the implemented ROM preserves skew-symmetry in the \emph{same} inner product as the energy estimate.
Or it does not.

If it does not, the argument collapses at its first line.
The discrete convective operator can inject energy.
This occurs in at least three standard situations.

\begin{enumerate}[label=(\roman*)]
\item \textbf{Non-divergence-free bases.}
If the ROM basis is not exactly solenoidal, the discrete convective term is not skew in the kinetic-energy inner product.
Then the identity $\ip{B(v,v)}{v}=0$ fails.
Energy can grow even when the full-order model is dissipative.
This is precisely why structure-preserving ROMs enforce solenoidality by construction \cite{BennerGoyal2022StructurePreservingROMSIAM,KleinSanderse2024EnergyConservingHyperReductionJCP}.

\item \textbf{Hyper-reduction that breaks structure.}
Even if the full-order convective discretization is energy preserving, a sampling or quadrature approximation can destroy the skew identity.
Then the reduced nonlinearity no longer satisfies $\ip{B_n(v,v)}{v}=0$.
Energy stability is lost.
Energy-conserving hyper-reduction is designed to avoid exactly this defect \cite{KleinSanderse2024EnergyConservingHyperReductionJCP}.

\item \textbf{Learned operators without dissipativity constraints.}
Operator inference can fit dynamics while producing a vector field that is not dissipative in the relevant metric.
Then the ROM ODE has no Lyapunov function inherited from the PDE.
It can be stable in-sample and unstable out-of-sample.
State-constrained and structure-preserving variants impose constraints precisely to remove this logical gap \cite{KramerPeherstorferWillcox2024OpInfSurveyARFM,KimKramer2025StateConstrainedOpInfJCP}.
\end{enumerate}

On the contrary, if the implemented ROM \emph{does} preserve skew-symmetry and if the closure is dissipative, then one obtains the discrete energy inequality of Theorem~\ref{thm:rom_energy}.
This is the correct replacement for the informal claim.
It replaces an assertion by an inequality with a failure flag.
\end{remark}

%% file: sections/11_conclusion.tex
% sections/11_conclusion.tex  (NEXT Output: Conclusion; embedded citations)

\section{Conclusion}\label{sec:conclusion}

We return to the primary proposition of Section~\ref{sec:primary}.
A reduced model for a dissipative flow is acceptable only if it carries a dissipation certificate.
The certificate must be stated as an inequality with computable constants.
This is the only form in which out-of-regime failure can be declared without ambiguity \cite{Sanderse2020NonlinearlyStableROMJCP,KleinSanderse2024EnergyConservingHyperReductionJCP}.

The first result is the energy certificate.
Definition~\ref{def:cert_scheme} fixes a class of time steps that preserve the energy method.
Theorem~\ref{thm:rom_energy} then yields a discrete energy inequality with explicit dependence on the forcing and viscosity.
The proof uses only two structural axioms: skew-symmetry of convection and dissipativity of closure.
If either axiom is violated by hyper-reduction or learning, the inequality can fail.
Remark~\ref{rem:sharpness} records this failure mechanism \cite{IliescuWang2021RegStabROMSIAM,AhmedPawarSanRasheedIliescuNoack2021ROMClosuresPoF,KimKramer2025StateConstrainedOpInfJCP}.

The second result is an error bound driven by a computable residual.
Definition~\ref{def:residual} fixes the residual in the dual norm demanded by the energy argument.
Theorem~\ref{thm:aposteriori} gives an a posteriori estimate in $\mathsf{X}(0,T)$ with an explicit Gr\"onwall factor depending only on computable Lipschitz constants on a declared regime set.
Theorem~\ref{thm:apriori} gives a conditional a priori rate under an explicit regime Lipschitz hypothesis.
Both results separate what is approximation from what is stability.
This separation is the point at which many informal ROM claims fail \cite{DoeringGibbon2022AppliedAnalysisNS,KramerPeherstorferWillcox2024OpInfSurveyARFM}.

The transition indicators of Section~\ref{sec:transition} are not detectors.
Proposition~\ref{prop:energy_barrier} yields a sufficient energy barrier.
Proposition~\ref{prop:enstrophy_thresh} yields an invariant enstrophy ball in a setting where the enstrophy budget closes.
Proposition~\ref{prop:resolvent_proxy} yields a resolvent-norm proxy for transient growth.
Each statement is an inequality with explicit inputs.
Each can fail in a declared way.
The distinction between a sufficient barrier and a predictive diagnostic is preserved.
This is consistent with the energy-stability perspective and with the modern resolvent literature \cite{Kerswell2022EnergyStabilityReview,RolandiRibeiroYehTaira2024ResolventInvitationTCFD,MarkeviciuteKerswell2024ThresholdTransientGrowthJFM}.

The FSI results are stated in the same logic.
Theorem~\ref{thm:fsi_wellposed} gives existence in a declared regime where the coupled energy closes, and uniqueness in a restricted regime.
This isolates a verifiable parameter class before numerics is discussed \cite{Canic2021MovingBoundaryProblemsBullAMS,BenesovaKampschulteSchwarzacher2023VariationalFPSINonRWA}.
Theorem~\ref{thm:fsi_discrete_energy} and Proposition~\ref{prop:margin} then convert the coupling into a discrete energy inequality with an explicit added-mass coefficient and an explicit stability margin.
The margin depends on density ratio and discrete trace/stiffness constants, as expected from Robin-type partitioned analyses \cite{ParrowBukac2026RobinRobinFPSI_JSC,GuoLinWangYueZheng2025RobinRobinExplicitFSI_arXiv,BukacFuSeboldtTrenchea2023TimeAdaptivePartitionedJCP}.

Finally, Section~\ref{sec:protocols} states the reporting protocol.
The protocol is part of the mathematics.
A bound without the constants that make it decidable is not a bound.
A stability claim without the skew and dissipativity defects is not a claim.
A computation outside the margin regime is not covered by the theorem.
Section~\ref{sec:scope} records the exclusions and the precise failure mode of a common informal argument.

What remains is clear.
The certificate program yields theorems in regimes where constants are computable.
Outside those regimes, the correct outcome is not a weaker theorem.
It is an explicit failure flag.

\section*{Acknowledgements}

The authors gratefully acknowledges the encouragement and research environment  provided by Dr.\ Ramachandra R.\ K., Principal, Government College (Autonomous), Rajahmundry, Andhra Pradesh, India.  
The author declares no funding, no conflict of interest, no human or animal ethics involvement, and no associated datasets.